\documentclass[a4paper, hidelinks, 11pt]{article}
\usepackage{graphicx} % Include figure files
\usepackage{amsmath, relsize} 
\usepackage{todonotes} 
\usepackage{amsthm} 
\usepackage{amsfonts}
\usepackage{amssymb}
\usepackage{url}
\usepackage{mathrsfs}
\usepackage{amsmath}
\usepackage{amssymb}
\usepackage{amsthm}
\usepackage{bbm}
\usepackage{comment}
\usepackage{caption} % Alexandra had to add this package in her file
\usepackage{subcaption}
\usepackage{tocloft}

\usepackage{enumerate}
\usepackage{setspace}
\usepackage{graphicx,color,hyperref}
\usepackage[margin=0.55in]{geometry}
\parskip \medskipamount
\newtheorem{theorem}{Theorem}[section]
\newtheorem{proposition}[theorem]{Proposition}
\newtheorem{corollary}[theorem]{Corollary}
\newtheorem{lemma}[theorem]{Lemma}
\theoremstyle{definition}
\newtheorem{definition}[theorem]{Definition}
\newtheorem{remark}[theorem]{Remark}
\newtheorem{example}[theorem]{Example}

\def\cA{{\mathcal A}}

\newcommand{\C}     {\mathbb{C}} 
 
\newcommand{\R}     {\mathbb{R}} 
\newcommand{\Z}     {\mathbb{Z}} 
\newcommand{\N}     {\mathbb{N}} 
\renewcommand{\P}   {\mathbb{P}} 
 
\newcommand{\E}     {\mathbb{E}}

\newcommand{\Acal}   {{\mathcal A }}
\newcommand{\Bcal}   {{\mathcal B }}

\newcommand{\Ecal}   {{\mathcal E }}

\newcommand{\Lcal}   {{\mathcal L }} 
\newcommand{\Mcal}   {{\mathcal M }}

\newcommand{\Pcal}   {{\mathcal P }} 
 
\newcommand{\Rcal}   {{\mathcal R }}

\newcommand{\Wcal}   {{\mathcal W }}

\newcommand{\Zcal}   {{\mathcal Z }}

\newcommand\numberthis{\addtocounter{equation}{1}\tag{\theequation}}

\makeatletter
\newcommand*{\biggg}[1]{{\hbox{$\left#1\vbox to20.5\p@{}\right.\n@space$}}}
\newcommand*{\Biggg}[1]{{\hbox{$\left#1\vbox to23.5\p@{}\right.\n@space$}}}
\makeatother

\usepackage{authblk}
\title{Coexistence,  enhancements and short loops  \\ 
	in random walk loop soups}
\author[1]{Nicolas Forien}
\author[2]{Matteo Quattropani}
\author[3]{Alexandra Quitmann}
\author[4]{Lorenzo Taggi}
\affil[1, 2, 4]{\small{Sapienza Universit\`a di Roma, Dipartimento di Matematica, Roma, Italy.}}
\affil[3]{\small{Weierstrass Institute for Applied Analysis and Stochastics, Berlin, Germany.}}
\date{}   

\numberwithin{equation}{section}

\date{
	\today
}

\begin{document}
	
	\maketitle

\begin{abstract}
We consider a general random walk loop soup which includes, or is related to, several  models of  interest, such as  the Spin O(N) model, the double dimer model and the Bose gas.
The  analysis of this model is challenging because of the presence of   spatial interactions between the  loops. 
For this model 
it is known from   \cite{QuitmannTaggi} that macroscopic loops occur in dimension three and higher when the inverse temperature is large enough.
Our first result is that, on the $d$ dimensional lattice, the presence of repulsive interactions is responsible for  a shift of the critical inverse temperature,  which is 
\textit{strictly greater} than $\frac{1}{2d}$,  the  critical value in the non interacting case. 
Our second result is  that a positive density of microscopic loops exists for all values of the inverse temperature.  This implies that, in the regime  in which  macroscopic loops are present,
microscopic and macroscopic loops \textit{coexist}. 
Moreover, we show that, 
even though the increase of the inverse temperature leads to 
an increase of the total loop length, the density of microscopic loops 
\textit{is uniformly bounded from above} in the inverse temperature. 
Our last result is confined to the special case in which the random walk loop soup is the one associated to the Spin O(N) model with arbitrary integer values of $N \geq 2$ and states that, on $\mathbb{Z}^2$,
the probability that two vertices are connected by a loop \textit{decays at least polynomially fast} with their distance. 
\end{abstract}

\section{Introduction and main results}
Random walk loop soups (RWLS) are intriguing  mathematical objects
which appear in various contexts in probability theory and mathematical physics. 
Our paper considers random walk loop soups on graphs in the presence of    spatial interactions.
These  are not only mathematically interesting,  but also physically relevant.
Indeed, they are a description of important  statistical mechanics models,  such  as the Spin O(N) model and  the  Bose gas.

The non interacting case has been studied first by Lawler and Werner
in 
\cite{LawlerWerner} in the context of two dimensional Brownian motion
and  actively researched both in continuous and discrete space partly due to its connections to the 
Gaussian free field and to  the Schramm-Loewner Evolution  (see e.g. \cite{BrugLis, LawlerFerrares, WernerSpatial}),
and as a percolation model  (see e.g. \cite{AlvesSapozhnikov, ChangSapozhnikov}). 

The  presence of  spatial interactions  between the loops affects significantly the phenomenology  and makes the rigorous analysis of such model even more challenging.

The model depends on a parameter, the inverse temperature,    $\beta \in [ 0, \infty)$.  Higher values of the inverse temperature favour
the total loop length. 
The most important questions involve the understanding of the distribution of the loop length as one varies the inverse temperature parameter. 
The purpose of the present paper is to make progress in answering such questions. 
As mentioned above,   this is not only interesting \textit{per se},
but  also allows us to introduce new methods
for the   analysis of the physical models that 
such a random walk loop soup describes.

Our first result, Theorem \ref{theo:enhancements} below,  compares the critical threshold of the interacting case on $\Z^d$,  with the critical threshold of the non interacting one,  that is $\frac{1}{2d}$.
We show that the  former is strictly greater than the latter. 
This is a consequence of the repulsive nature of the interactions: while in the non interacting case each random  loop is `free to explore'  the whole space,  in the interacting case
it is only allowed to visit  regions which are `not too much filled' by the other loops 
(the repulsion would be too strong otherwise).
The simple proof of our first theorem consists in showing that the reduction of the amount of available space for each loop leads to a entropy loss  and thus to a shift of the critical threshold.

While it is easy to show (using classical methods, such as the cluster expansion) that the loops are  `localized' when the inverse temperature is small enough,  understanding the regime of large values of $\beta$ is much harder.   For large values of the inverse temperature one expects the occurrence of \textit{macroscopic loops}  (i.e., loops whose length grows proportionally to the volume of the system) in dimension three and higher.
This fact has been proved in  \cite{QuitmannTaggi}
using the reflection positivity  method \cite{Frohlich, L-T-first,  T}.
% (remains a major open problem providing a more general proof which does not use reflection positivity,  see \cite{GarbanSpencer} for some recent results in this direction in the context of spins)
Providing more information on the distribution of the loop length
is of great interest. 

Our second result,  Theorem  \ref{theo:coexistence} below,
goes in this direction and
provides additional information on the length distribution  of  \textit{microscopic loops}, i.e., those loops whose 
length does not depend on the size of the system. 
Firstly, our theorem implies that,  for each value of the inverse temperature,
uniformly in the volume,  there exists a positive \textit{density
of microscopic loops} of any given length
(i.e.  the expected number of loops of a given length divided by the volume,  see  equation (\ref{eq:defdensity}) below).
Our theorem thus implies that, in the regime 
in which macroscopic loops are known to occur (corresponding to large values of the inverse temperature)
macroscopic and microscopic loops \textit{coexist}.
Secondly, our theorem states that,  uniformly in the volume, the density of  microscopic loops of any given length is \textit{uniformly
bounded from above} in the inverse temperature. 
In other words, even though arbitrarily large values of the inverse temperature lead to arbitrarily large 
number of visits at each vertex, 
the density of  microscopic loops of any given length  
does not grow to infinity with the  inverse temperature. 
This suggests that only the longer loops
are affected by the increase of the inverse temperature.
We consider this result and the analysis which we develop for   its proof
the most interesting part of our paper
(see also the discussion  right after Theorem \ref{theo:nomacroscopicloops} for further comments).

While the previous results are valid for (or can easily be adapted to)  a quite general choice of the interaction,
our last result, Theorem \ref{theo:nomacroscopicloops} below,  is specific for the   random walk loop soup describing the Spin O(N) model.
It is a well known consequence of the Mermin-Wagner theorem that, on $\Z^2$, the two-point function in the Spin O(N) model decays at least polynomially fast with the distance between the vertices,  for any integer value of $N \geq 2$ and for each value of the inverse temperature. 
We show that, under the same conditions, also the probability that two vertices
are connected by a loop decays at least polynomially
fast with the distance between the vertices. 
This implies that, for any value of the inverse temperature, macroscopic loops \textit{do not occur} in two dimensions.
%Our result is unfortunately not sharp.  Indeed, when  $N>2$,
%one expects the decay of this probability to be 
%exponential with the distance for any value of the inverse temperature.

%In the non-interacting case,  the model is well defined only for values of $\beta \in (0, \frac{1}{2d}]$.  When $\beta \in (0,  \frac{1}{2d})$ only short loops are present.  The  case $\beta = \frac{1}{2d}$ is critical and most interesting,  in this case the random walk loops exhibit diffusive behaviour. 

%One of the most challenging questions is whether a phase transition in the model
%occurs as one varies the inverse temperature parameter,  $\beta \in [0, \infty)$.

%Increasing the value of $\beta$ favours the total loop length.  
%Because of the repulsive nature of the spatial interactions,  however,
%the  loops tend to avoid each other and themselves and it is not....

%The purpose of the present paper is to investigate 

%Our first simple theorem involves the comparison of the critical threshold

\subsection{Definitions} \label{sect:definitions}

Let $G = (V,E)$ be a finite undirected graph. 
For each pair of vertices $x, y \in V$, we let 
$\mathcal{L}(x,y)$ be the set of walks from $x$ to $y$,  i.e., 
finite ordered sequences of vertices in $V$,
$\ell = \big (\ell(0), \ell(1),   \ldots \ell(k) \big )$,
such that $\ell(i)$ is a nearest-neighbour of $\ell(i-1)$ for each 
$i \in \{1, \ldots, k\}$,   $\ell(k) =y$,  $\ell(0)=x$ and $k \geq 1$. 
For any such sequence $\ell = \big (\ell(0), \ell(1),   \ldots, \ell(k) \big )$,  we denote by $| \ell | : = k$ 
the \textit{length} of the walk $\ell$. 
We define $\mathcal{L} := \cup_{x \in V} \mathcal{L}(x,x)$ and we call \textit{rooted oriented loop} {(r-o-loop)}
any element $\ell \in \mathcal{L}$. 
We let $\Omega := \cup_{n=1}^ \infty \mathcal{L}^n \cup \{ \boldsymbol{0}
 \}$ be the configuration space of the random walk loop soup, whose elements are ordered collections of rooted oriented loops, and 
 $\boldsymbol{0} \in \Omega$ is the  configuration with no loop.
Given any configuration $\omega \in \Omega$, we denote by $|\omega|$ the \textit{number of loops} in $\omega$, i.e., $|\omega|$ is defined as the  integer $n$ such that $\omega \in \mathcal{L}^n$ for  $n >0$ and $| \boldsymbol{0} | = 0$. 
For any $\omega =(\ell_1,\,\ldots,\,\ell_{|\omega|}) \in \Omega$,
we define by
$$
n_x(\omega) := \sum\limits_{n=1}^{|\omega|} \sum\limits_{j=0}^{|  \ell_n  | -1} \mathbbm{1}_{
\{ \ell_n(j) = x \} }
$$
the \textit{local time} at $x \in V$.
Moreover, we  consider a    \textit{weight function},
$U : \mathbb{N}_0 \rightarrow \mathbb{R}_0^+$,
which  weighs the local time at sites and may, for example, 
suppress configurations with local time above a certain threshold.
We also introduce two parameters $\beta, N \in \mathbb{R}_0^+$
and define a probability measure in  $\Omega$. 
Our measure  assigns to any realisation $\omega \in  \Omega$, 
$\omega = \big  ( \ell_1, \ell_2, \ldots \ell_{|\omega|} \big ) \in \Omega$, 
the weight
\begin{equation}\label{eq:measure}
\mathcal{P}_{G, U,  N, \beta} (\omega ) : = \frac{1}{\mathcal{Z}_{G,  U,  N,  \beta}}
\, \, \frac{1}{|\omega|!}  
{\bigg(\frac{N}{2} \bigg)}^{|\omega|}
\prod_{ i=1 }^{|\omega|}
\frac{\beta^{ |\ell_i|}}{|\ell_i|}\,
\prod_{ x \in V} U\big(n_x(\omega) \big)\,,
\end{equation}
where  
\begin{equation}\label{eq:partition}
\mathcal{Z}_{G,  U,  N,  \beta} =  1  + 
\sum\limits_{n=1 }^{\infty}  \, \, \frac{1}{n!} 
{\bigg(\frac{N}{2} \bigg)}^{n} 
\sum\limits_{{(\ell_1 \ldots, \ell_n) \in \mathcal{L}^n}}
\, \,   \prod_{ i=1   }^n
\frac{\beta^{ |\ell_i|}}{|\ell_i|} 
 \prod_{ x \in V} U \big (n_x(\omega)  \big )
\end{equation}
is the so-called    \textit{partition function}.
If there exists some $M\in\N$ such that the weight function $U: \N_0\to[0,\infty)$ satisfies
\begin{equation} \label{eq:fast decayingweightfunction}
	U(n)\le \frac{M}{n}U(n-1)\,  \qquad\forall n\in\N\,, 
\end{equation}
we say that $U$ is $M$-\textit{decaying}, 
and we say that it is \textit{fast decaying} if it is $M$-decaying for some $M \in \mathbb{N}$. 
Under condition \eqref{eq:fast decayingweightfunction}, we can ensure that  (\ref{eq:partition}) is finite  and thus that the measure (\ref{eq:measure})
is well defined  for each $\beta \geq 0$,  see \cite[Lemma 2.2]{QuitmannTaggi}. We denote by $\Ecal_{G,  U,  N,  \beta}$ the expectation under the measure \eqref{eq:measure}.

\subsubsection{Special cases}
We now discuss special cases and connections to other statistical mechanics models. 

\vspace{0.2cm}

\textit{The non interacting case.}
When $U(n) = 1$ for each $n \in \mathbb{N}_0$,  our model reduces to the random walk loop soup which has been considered in   \cite{AlvesSapozhnikov,  BrugLis, ChangSapozhnikov,  LawlerFerrares, WernerSpatial},
in which no spatial interaction between the loops is present.
This model is also referred to as Poisson loop ensemble in 
\cite{ChangSapozhnikov, Lejan}.
To make the correspondence more clear,  we recall the definition
of the Poisson loop ensemble, following \cite{ChangSapozhnikov, Lejan}.
Two rooted oriented loops are said to be  \textit{cp-equivalent} if they coincide 
after a cyclic permutation.
Equivalence classes of rooted oriented loops are called 
\textit{unrooted loops.}
We  weigh  each rooted oriented loop $\ell = (x_0,  x_1, \ldots, x_k) \in \mathcal{L}$
through the measure,
$$
\hat{\mu} (\ell) = \frac{1}{k} \left( \frac{1}{1 + \kappa}\right)^k\,,
$$
where $\kappa > - 1$. 
The push-forward of $\hat{\mu}$  on the space of unrooted loops is denoted by  $\mu_{\kappa}$.
For $\alpha \geq 0$ the Poisson loop ensemble of intensity $\alpha \mu_{\kappa}$,
denoted by $L_{\alpha, \kappa}$, 
is a random countable collection of unrooted loops such that the point measure
\smash{$
\sum_{\ell \in L_{\alpha, \kappa}} \delta_\ell
$}
is a Poisson random measure of intensity $\alpha \mu_\kappa$ (here, $\delta_\ell$ is the Dirac mass at the unrooted loop $\ell$).
This model corresponds to the random walk loop soup (\ref{eq:measure}) 
in the special case $U(n) = 1$ for each $n \in \mathbb{N}_0$,
when $\frac{N}{2}=\alpha$  and \smash{$\beta =  \frac{1}{ 1 + \kappa}$}. 
The random walk loop soup considered in \cite{LawlerFerrares}
corresponds to the special (critical) case $\beta = \frac{1}{1 + \kappa} = \frac{1}{2d}$. 
The equivalence of the definitions can be deduced from \cite[Section~3]{QuitmannTaggi}, where also a  third equivalent  formulation (which is recalled later also in the present paper) is presented.

\vspace{0.2cm}

\textit{The Spin O(N) model.}
The Spin $O(N)$ model is one of the most important statistical mechanics models.
Special cases of the Spin $O(N)$ model are the Ising model,  the XY model,  and the Heisenberg model corresponding
to $N=1,2$ and $3$,  respectively. 
It is a 
spin system with spins taking values in the $N-1$ dimensional unit sphere, $\mathbb{S}^{N-1}$.
The configuration space is the product space $\Omega = (\mathbb{S}^{N-1})^V$,  and the model is defined through the following expectation operator
$$
\langle f  \rangle_{G, N, \beta}   = \frac{1}{Z^s_{N, \beta}} \int_{\Omega} 
d \boldsymbol{S} f( \boldsymbol{S} ) \exp{  \Big (  - \beta \sum\limits_{\{x,y\} \in E   } S_x \cdot  S_y  \Big ) }\,,
$$
where $f : \Omega  \mapsto \mathbb{R}$ is any measurable function, 
$d \boldsymbol{S} =  \prod_{x \in V} d S_x$ is the product of uniform measures in $\mathbb{S}^{N-1}$,
and $Z^s_{N, \beta}$ is a constant such that $ \langle 1  \rangle_{G,N, \beta}  = 1$. 
We use the notation $S_x = (S_x^1, \ldots, S_x^N)$. 
Our loop model is related to the Spin $O(N)$ model 
if~$N\in\mathbb{N}$ and $U = U^s_N$, where for each $n \in \mathbb{N}_0$, 
\begin{equation}\label{eq:SpinWeightfunction}
U^s_N(n) : = \frac{\Gamma(\frac{N}{2})}{{2^n} \, \Gamma( n + \frac{N}{2})}\,.
\end{equation}
The main connection  between the Spin O(N) model   and the random walk loop soup is given by the following equivalence
\begin{equation}\label{eq:spinconnection}
\forall x, y \in V \quad \quad   \langle S^1_x S_y^1 \rangle  = \frac{\mathcal{Z}_{G, U^s_N, N, \beta}(x,y)}{\mathcal{Z}_{G, U^s_N, N, \beta}}\,,
\end{equation}
which relates the two-point function, a central object in the analysis of the Spin O(N) model, to the partition function of a random walk loop soup with an `open loop' connecting $x$ and $y$, 
\begin{equation}\label{eq:partitionwithwalk}
\mathcal{Z}_{G, U, N, \beta}(x,y) := 
\sum\limits_{ \gamma \in \mathcal{L}(x,y) } \beta^{|\gamma|}
\sum\limits_{ \omega = (\ell_1, \ldots, \ell_{ |\omega| }) \in \Omega }
\, \, 
{\bigg(\frac{N}{2} \bigg)}^{|\omega|}
\, \,   \prod_{ i=1   }^{|\omega|}
\frac{\beta^{ |\ell_i|}}{|\ell_i|} 
 \prod_{ x \in V} U \big (n_x(\omega) +  n_x(\gamma) \big )\,,
\end{equation}
where $U = U^s_N$,  and 
 $n_x(\gamma) = \sum_{i=0}^{|\gamma|} \mathbbm{1}_{ \{  \gamma(i) = x \}  } $ is the local time of the walk $\gamma$ at $x$.
The identity (\ref{eq:spinconnection}) for $U = U^s_N$  was proved in  \cite{BFS1} (see also the Appendix of \cite{BFS2} for the correction of a mistake in the definition of the model in  \cite{BFS1}).

\textit{The Bose gas.}
The Bose gas is another important statistical mechanics model. 
The model is defined in the functional analytic framework of quantum mechanics
(we refer to \cite{LiebBook} for its definition and an overview) and can be reformulated as a random walk loop soup in discrete or continuous space using  the  Feynman-Kac formula
\cite{Ginibre}.
One of the most important (unsolved) mathematical problems involving the Bose gas is the proof of the occurrence of a phase transition,   known as Bose-Einstein condensation,
 which can be related to the occurrence
of macroscopic loops in the corresponding random walk loop soup 
(this relation was suggested by Feynman, \cite{Feynman},  and has not been verified 
from a mathematically rigorous point of view). 
For this reason in recent years  significant effort has been made to understand basic properties of interacting random walk loop soups, see e.g.   \cite{Betz0, BU1, BU3, Bogachev, D-V, ElboimPeled}, none of these papers however considers
random loops in the presence of spatial interactions
 like in the Bose gas or in the present paper. 
The model we consider  is a simplification of the Bose gas in the grand canonical ensemble with (short range)  interaction potential  $v = \alpha \,  \delta_0 $, with $\alpha>0$,
 if the weight function satisfies  $U(n) = \exp( - \alpha   \binom{n}{2} )$,
 with the parameter $\beta$ playing the role of the fugacity 
(see \cite[Appendix A1]{QuitmannTaggi} for a discussion on the connection between our model and the Bose gas). 
The methods we develop can be applied in wider generality and adapted directly to the random loop representation of the Bose gas in discrete space,
even in the presence of long range interactions
(for clarity and conciseness reasons, we go for a simpler setting, which is also more general).

\textit{Other models.}
Other relevant models which belong or are related to the general framework of our random walk loop soup are the double dimer model \cite{QuitmannTaggi2}, corresponding to the case $U(n) = \mathbbm{1}_{n=1}$,
and the loop O(N) model (see \cite{PeledSpinka} and references therein for an overview), which is related to the case  $U(n) = \mathbbm{1}_{n\leq 1}$. We refer to \cite{QuitmannTaggi} for a discussion on these connections. 
Unfortunately not all our results (but only some of them) can be applied
to the study of such cases,  since for conciseness  reasons it is practical for us to make the assumption that $U(n)>0$ for any $n \in \mathbb{N}_0$, in order to avoid hard-core constraints.

\subsection{Main results}
We now let $G= (V,E)$ be a infinite connected graph {with bounded degrees} and for each finite subset $K \subset V$ we define the graph $G_K = (K, E_K)$, where $E_K$ is defined as the subset of $E$ whose elements have both end-points in  $K$.  
Given any two vertices $x, y \in V$,  we  denote by $\{x \leftrightarrow y\}$ the event that there exists a loop visiting both $x$ and $y$.
We say that the random walk loop soup admits exponential decay of the connection probability at $(\beta, N)$ (in short:  it admits \textit{exponential decay})  if there exists $c \in (0, \infty)$ such that, for each finite subset $K \subset V$, for each $x, y \in K$,  we have that,
\smash{$\mathcal{P}_{G_K,  U, N,  \beta}(x \leftrightarrow y  ) \leq  e^{-c\,d(x,y)}$},
where $d(\cdot , \cdot )$ denotes the graph distance in $G$. 
We define for each $N>0$ the \textit{critical inverse temperature} as,
\begin{equation}\label{eq:criticalthreshold}
\beta_c(U, N) := \sup\{  \beta \geq 0 \, : \,  \mbox{ exponential decay at $(\beta, N)$}   \}\,,
\end{equation}
which might be infinite. 
It is not hard to show that the critical value of the random walk loop soup in the non-interacting case is $\frac{1}{2d}$. 
Indeed,  in this case the 
model admits exponential decay if $\beta < \frac{1}{2d}$, while it does not  admit exponential decay for $\beta = \frac{1}{2d}$ (and it is not well defined for $\beta > \frac{1}{2d}$, since the partition function is infinite). 
Our theorem shows that introducing any form of repulsion (through a fast decaying weight function) satisfying quite general conditions increases 
 the critical threshold strictly. 
 
 We say that the weight function $U$ is \textit{nice} if 
 it is fast decaying,  if 
 $U(n) \leq U(0)$ for every $n \in \mathbb{N}$,  and if it is submultiplicative,  i.e.,  $U(n+n^\prime) \leq U(n) U(n^\prime)$ 
 for every $n,\,n' \in \mathbb{N}$ (see Section \ref{sect:proofenhancements} for some examples of nice weight function, 
including the Bose gas and the Spin O(N) model). 
\begin{theorem}
\label{theo:enhancements}
Let $N>0$, suppose that $U$ is nice,  and consider the random walk loop soup in $\mathbb{Z}^d$. Then, 
\begin{equation}
\beta_c(U, N)  > \frac{1}{2d}\,. \label{eq:fact1}
\end{equation}
\end{theorem}
Similar enhancement theorems have been proved also for the  lattice permutation model \cite{Betz0, Betz2},   and for the loop O(N) model
\cite{TaggiShifted, Glazman2}, using methods which are specific for these models. 

Our proof of Theorem \ref{theo:enhancements} actually provides a quantitative estimate on the difference in  \eqref{eq:fact1}, depending on the choice of the function $U$. Nevertheless, we do not expect this estimate to be sharp but it can be rather improved for specific choices of the weight function. In Remark \ref{rmk:thm1.1-quant} we comment further on this point.

Fix a finite  bipartite graph $G=(V,E)$, a fast decaying function $U:\N_0 \to \R^+$ and parameters $N,\beta \in \R^+$. To state our next main theorem we  introduce the\textit{ density of loops} of  length $k \in \mathbb{N}$,
	\begin{equation}\label{eq:defdensity}
	\rho_{G, U,N,\beta} (k) :=  \frac{1}{|V|}\,\mathcal{E}_{G,  U, N, \beta}  \Big (  \sum\limits_{  i = 1  }^{|\omega|} \mathbbm{1}_{  \{  | \ell_i | = k   \}  } \Big )\,,
	\end{equation}
{where we recall from Section \ref{sect:definitions} that we denote by $\mathcal{E}_{G,  U, N, \beta}$ the expectation under the measure \eqref{eq:measure}.}
Note that
the numerator in the right-hand side of (\ref{eq:defdensity}) corresponds to  expected number of loops of length $k$.
Despite we use the word `density',  $\rho_{G, U,N,\beta} (k)$  is not necessarily bounded from above by   one (rather than dividing by the total number of loops, which is proportional to $|V|$, we divide by $|V|$).
	Our next theorem states that the density of loops of any given length $k \in 2 \mathbb{N}$ is strictly positive and uniformly bounded from above in $\beta$, regardless of  $|V|$. 
Moreover, for every choice of $\beta>0$ we also deduce a lower bound for the same quantity, which does not depend on $|V|$.
	\begin{theorem}
		\label{theo:coexistence}
		Let $\beta,N>0$ and let $U$ be a fast decaying weight function  such that 
		$U(n)>0$ for all $n\in\N_0$.
		For any  $k\in 2\mathbb{N}$ there exist $c_1=c_1(k,\Delta,N)<\infty$ and $c_2=c_2(k,\Delta,\beta,N,U)>0$ such that, for every finite connected
		bipartite graph  $G = (V,E)$ with maximal degree $\Delta$, 
		we have that,
		\begin{equation}\label{eq:rhobound}
		c_2\le \rho_{G,U,N,\beta} (k) \le c_1\,.
		\end{equation}
	\end{theorem}
 Our theorem has two  main implications.  The lower bound implies that, even in the regime of large values of $\beta$,   in which macroscopic loops occur \cite{QuitmannTaggi} 
(i.e., loops whose length is proportional to $|V|$),
 there remains  a density  (depending on $\beta$) of microscopic loops (i.e.,  loops whose length 
is independent from $|V|$) of any length.
In other words, microscopic and macroscopic loops \textit{coexist}. 
On the other hand, the main interest of our upper bound is in the fact that  it holds uniformly in $\beta$. 
It implies that,  even though increasing $\beta$ to infinity has the effect of increasing to infinity 	the local time at each site uniformly in the size of the system
	(this follows from  \cite[Lemma~5.6]{QuitmannTaggi}),
	the density of microscopic loops whose length is less than any finite   $k$
	 is  bounded from above uniformly in $\beta$ and in the size of the system. 
	In other words,  the increase of the local time as $\beta$ gets larger is only due
	to an increase in the number of the longer loops.

	\vspace{0.2cm}
	
\textbf{Discussion. }
The techniques that we use to prove our Theorem \ref{theo:coexistence} can be applied in wider generality.  For example,  they can be adapted to the  random walk loop soup representing the Bose gas, even in the presence of long range interactions
(we consider a simpler setting since it is more general and also for clarity reasons).
It is of great interest introducing new methods for the analysis of such models.
Indeed, {except for the special case of the hard-core Bose gas under the so-called half-filling condition \cite{DysonLiebSimon, KennedyLiebShastry},} no proof of occurrence of a phase transition in the Bose gas in dimension $d > 2$ is known.  Also  for the random walk loop soup (\ref{eq:measure})
and for its special case,  the Spin O(N) model with $N>2$, 
 the  only known proof of occurrence of a phase transition when $d>2$ uses reflection positivity,
which is a quite limited method.  For example it only works on the torus of $\mathbb{Z}^d$
with even side length and for certain classes of interactions.
It is of great interest providing a  more general proof of occurrence of a phase transition.
The derivation of the (uniform in $\beta$) upper bound  (\ref{eq:rhobound}) can be considered a progress towards this direction. 
For concreteness,  let  $G_L = (\Lambda_L,E_L)$ be the graph with vertex set $\Lambda_L:=\big\{(x_1,\dots,x_d) \in \Z^d \, : \, x_i \in (-\frac{L}{2},\frac{L}{2} ] \,  \, \,  \forall i \in [d]\big\}$ and edges connecting nearest-neighbour vertices and assume that $d>2$.
For any fast decaying $U:\N_0 \to \R^+$ and $N,\beta \in \R^+$ consider the average local time of the microscopic loops
\begin{equation}\label{eq:localtimemicro}
T_{U,N}(\beta) :=  \sum_{k=2}^{\infty}  k \limsup_{L  \rightarrow \infty}  \,  \rho_{G_{2L},U,N,\beta}(k)\,.
\end{equation}
This quantity corresponds to the average number of times sites are visited by some microscopic loop (note that the limit  is taken after the sum, this ensures that only the contribution of the microscopic loops is taken into account).
One expects the  local time  of the microscopic loops to increase with $\beta$ up to reaching the critical threshold $\beta_c$. 
Any further increase of $\beta$ beyond such a threshold is expected to lead to an increase of the  local time of the macroscopic loops but not of the microscopic loops. 
Since we know from \cite{QuitmannTaggi} 
that the total (i.e., of the microscopic \textit{and} non-microscopic loops)
local time  increases to infinity
with $\beta$,  providing an upper bound for  (\ref{eq:localtimemicro})
which is finite and  uniform in $\beta$  would imply the existence of 
non-microscopic loops,  and, thus,  the occurrence of a phase transition.

Our Theorem \ref{theo:coexistence} provides an upper bound which is uniform in $\beta$ for each term in the sum in (\ref{eq:localtimemicro}), 
and we consider this  an interesting progress.
Unfortunately there is still an important gap to fill to show that \textit{the whole sum} 
(and not only any single term which appears in the sum) is  bounded from above  by a finite constant uniformly in $\beta$, thus obtaining a new general proof of occurrence of a phase transition.
{
Indeed, 	the constant $c_1$ in Theorem~\ref{theo:coexistence} is far from optimal and is exponentially growing in $k$ (see equation (\ref{eq:proofupperbound}) below for a quantitative estimate.).
}

We now present our third and last main result. 
 It is well known from the Mermin-Wagner theorem 
 that the spin-spin correlation in the Spin O(N) model,
$\langle  S_x^1 S_y^1 \rangle$, decays at least polynomially fast with
$|y-x|$ on $\mathbb{Z}^2$,
for each integer $N \geq 2$
(see e.g. \cite[Section 9.4]{FriedliVelenik}).
However,   (\ref{eq:spinconnection})
does not allow us to deduce from this result that 
also the probability that $x$ is connected to $y$ by some loop
in the random walk loop soup exhibits the same decay. 
Our last theorem fills this gap and,  in turn, it implies that no macroscopic loop
exists in dimension two. 
\begin{theorem}
\label{theo:nomacroscopicloops}
Let $d=2$ and recall the definition of $G_L$. 
Then, for any $\beta \geq 0$  and integer $N \geq 2$   there exists $c(\beta, N) \in (0, \infty)$ such that for any $x,y \in \Z^2$, for any $L$ large enough,
$$
  \mathcal{P}_{G_L, U^s_N, N, \beta} (x \leftrightarrow y) \leq d(x,y)^{-c}\,.
$$
\end{theorem}
It may well be that the same result can be deduced from the
loop representations of the Spin O(N) model which 
have recently been used in \cite{AizenmanPeled,  Lis}, 
and results on the absence of macroscopic loops
in two dimensions for an analogous loop representation of the  Spin O(N) model have also appeared in \cite{Be-U}.
The techniques and results of these papers, however, are limited to the case $N=2$.

\section{Proof of Theorem \ref{theo:enhancements}}
\label{sect:proofenhancements}
In this section we present the proof of Theorem \ref{theo:enhancements}. More precisely, we  prove a more general {result}, Theorem \ref{theo:enhancements-new} below, from which Theorem \ref{theo:enhancements} follows as a special case.  Recall the definition of  \textit{nice} weight function $U$ right before Theorem~\ref{theo:enhancements}. Notice also that, if $U(0)=a>1$, we can consider an equivalent model by considering the weight function $\bar U(\cdot)=\bar U(\cdot)/a$, which is again nice and moreover satisfies $\bar U(0)=1$. Therefore, in what follows, without loss of generality we will assume that a nice weight function satisfies $U(0)=1$.
\begin{example} We now give further examples of  nice weight functions.
	\begin{itemize}
		\item The weight function $U_N^s$ of the Spin $O(N)$ model, defined in \eqref{eq:SpinWeightfunction}, for every integer $N\ge 2$.
		\item The \emph{factorial} weight function, defined as	\begin{equation}\label{eq:def-bar-U-s}
			{U_{\rm F}}(n)=\frac{1}{n!}\, .
		\end{equation}
		\item The \emph{pairwise weight function} of intensity $\alpha\in[0,\infty)$, $U_\alpha:\N_0\to (0,1]$, defined as
		\begin{equation}\label{eq:def-pairwise-potential}
			U_\alpha(n)=e^{-\alpha\binom{n}{2}}\ ,\qquad n\in\N_0\, ,
		\end{equation}
	for some $\alpha>0$.
	\end{itemize}
\end{example}

The next theorem states that the critical inverse temperature of the random walk loop soup with weight function $U$ is (weakly) bounded from below by the inverse connectivity constant of the weakly self-avoiding walk with the same weight function.
\begin{theorem}
	\label{theo:enhancements-new}
	Let $d\geq 1$ and $G=(V,E)$ be $\Z^d$. Fix $N>0$ and suppose that $U$ is nice. 
	Then, 
	\begin{equation}\label{eq:def-tilde-beta}
		\beta_c(U,N)\ge \tilde{\beta}_c(U):=\sup\left\{\beta\ge 0\:\:{\rm s.t. }\:\limsup_{k\to\infty}\frac{1}{k}\log\chi_{ U}(k)< -\log(2d\beta)\right\}\, ,
	\end{equation}
where, for $k\in\N$,
\begin{equation}\label{eq:def-chi}
	\chi_{ U}(k):= \mathbf{E}_{o}\left[ \prod_{x \in V} U(n^{(k)}_x(X))\right]\, ,
\end{equation}
and $\mathbf{E}_o$ denotes the expectation with respect to the law of a simple symmetric random walk $X = (X_t)_{t\ge 0}$ in $\Z^d$ starting at the origin, that is to say,
for every function~$f$ which depends only on the first~$k$ steps of the walk,
$$
E_o(f)
\ =\ 
\frac 1{(2d)^k}
\sum_{\substack{X=(x_0,\,x_1,\,\ldots,\,x_k)\in(\Z^d)^{k+1}\ :\\x_0=0\text{ and }|x_i-x_{i-1}|=1\ \forall i\in[k]}}
f(X)
\,,
$$
and $n_x^{(k)}(X)$ is the number of visits of the walk~$X$ at site~$x$ before time $k$, not counting the initial step, 
$$n_x^{(k)}(X)
:=
\big|\big\{j\in[k]\,:\,X_j=x\big\}\big|\,.
$$	
\end{theorem}
%
%From Theorem \ref{theo:enhancements-new} we can deduce the following corollary.
%\begin{corollary}\label{coro:sharpness}
%In the same setting of Theorem \ref{theo:enhancements-new}
%	$$	\beta_c(U) > \frac{1}{2d}\, .$$
%\end{corollary}

	Before entering into the details, it is worth highlighting that the proof of Theorem \ref{theo:enhancements-new} relies on a simple idea: exploiting the submultiplicativity of the function $U$, the contribution of distinct loops can be factorized so to obtain an effective upper bound on the probability of having a loop joining two distinct vertices.

We will later show how Theorem \ref{theo:enhancements} follows from Theorem \ref{theo:enhancements-new}.
\begin{proof}[Proof of Theorem \ref{theo:enhancements-new}]
Let $K \subset \mathbb{Z}^d$ be finite and let $G_K=(K,E_K)$ be the subgraph of $\mathbb{Z}^d$ with edges having both end-points in $K$, let $o,y \in K$ and set $r:=d(o,\,y)$.
Using the symmetry among the labels of the loops we have that, for any $\beta \in \R^+$,

	\begin{align}
		& \Pcal_{G_K,U,N,\beta}({o}\leftrightarrow y)
		%\le \Pcal(\exists i\le |\omega|\text{ s.t. }o\in\ell_i\ ,\ |\ell_i|\ge r) 
		\nonumber \le \sum_{n\ge 1}n\ \Pcal_{G_K,U,N,\beta}\big(|\omega|=n ,\ o\in\ell_1,\ |\ell_1|\ge r\big) \\
		\nonumber	&=\frac{1}{\Zcal_{G_K,U,N,\beta}}\sum_{n\ge 1}\frac{1}{(n-1)!}\left(\frac{N}{2} \right)^n\sum_{x_1,\dots,x_n\in {K}}\sum_{\substack{\ell_i\in\mathcal{L}(x_i,x_i)\\i=1,\dots,n}}\prod_{i=1}^n\frac{\beta^{|\ell_i|}}{|\ell_i|}\prod_{x \in K}U(n_x(\ell_1,\dots,\ell_n))\mathbbm{1}_{o\in\ell_1}\mathbbm{1}_{|\ell_1|\ge r}\\
		\nonumber	&=\frac{N}{2} {\sum_{z\in K}\sum_{\ell\in\mathcal{L}(z,z)}}\frac{\beta^{|\ell|}}{|\ell|}\mathbbm{1}_{o\in\ell}\mathbbm{1}_{|\ell|\ge r} \\
		\nonumber & \qquad \times \frac{1}{\Zcal_{G_K,U,N,\beta}}\sum_{n\ge 1}\frac{1}{(n-1)!}\left(\frac{N}{2} \right)^{n-1}\sum_{x_2,\dots,x_n\in {K}}\sum_{\substack{\ell_i\in\mathcal{L}(x_i,x_i)\\i=2,\dots,n}}\prod_{i=2}^n\frac{\beta^{|\ell_i|}}{|\ell_i|}\prod_{x \in K}U(n_x(\ell,\ell_2,\dots,\ell_n))\,.
	\end{align}
Using that $U$ is sub-multiplicative and exploiting the symmetry of with respect to root of $\ell$, we obtain
\begin{align*}
 \Pcal_{G_K,U,N,\beta}({o}\leftrightarrow y)		&\le\frac{N}{2} {\sum_{z\in K}\sum_{\ell\in\mathcal{L}(z,z)}}\frac{\beta^{|\ell|}}{|\ell|}\mathbbm{1}_{o\in\ell}\mathbbm{1}_{|\ell|\ge r}\prod_{x \in K} U(n_x(\ell))\\
		&\le\frac{N}{2} \sum_{\ell\in\mathcal{L}(o,o)}\beta^{|\ell|}\mathbbm{1}_{|\ell|\ge r}\prod_{x \in K} U(n_x(\ell))\\
	& \le  \frac{N}{2}\sum_{k\ge r}(2d\beta)^k\mathbf{E}_{o}\left[ \prod_{x \in V} U(n_x^{(k)}(X))\mathbbm{1}_{X_k=o}\right]\\
		\nonumber	&\le\frac{N}{2}\sum_{k\ge r}(2d\beta)^k\chi_{ U}(k)\, .
	\end{align*}

 Let now $\beta<\tilde{\beta}_c( U)$. The definition of $\tilde{\beta}_c( U)$ implies that there exist $c_1,c_2\in(0,\infty)$ such that
	$$(2d\beta)^k\chi_{ U}(k)\le c_2\, e^{-c_1 k},\qquad \forall k\in\N\, . $$
	Therefore, for some $c_3>0$
	\begin{equation}
		\mathcal{P}_{G_K,U,N,\beta}(o\leftrightarrow y)\le c_3\: N\:e^{-c_1 d(o,\,y)}\,,\qquad o,y\in V\, .
	\end{equation}
Hence, $\beta_c(U,N)> \beta$. Since $\beta<\tilde{\beta}_c(U)$ was chosen arbitrarily, this concludes the proof of the theorem.
\end{proof}

We are now ready to present the proof of Theorem \ref{theo:enhancements}. Essentially, given the result in Theorem \ref{theo:enhancements-new}, we only need to provide a positive lower bound for the difference between the inverse connectivity constant of the self-avoiding walk with weight $U$, $\tilde{\beta}_c(U)$, and that of the simple random walk, i.e., $\tfrac1{2d}$.

	\begin{proof}[Proof of Theorem \ref{theo:enhancements}]
		Let $G$, $N$, $U$ as in the statement of Theorem \ref{theo:enhancements-new}. By Theorem \ref{theo:enhancements-new} it {suffices} to show that $\tilde{\beta}_c( U)>\frac{1}{2d}$. From the definitions of $\tilde{\beta}_c( U)$ and of $\chi_{U}$ in \eqref{eq:def-tilde-beta} and \eqref{eq:def-chi}, it follows  that this is equivalent to showing that
		\begin{equation}\label{eq:limit-free-energy-bar}
			\limsup_{k\to\infty}\frac{1}{k}\log \mathbf{E}_o\left[\prod_{x \in V} {U}(n_x^{(k)}(X)) \right]<0\, .
		\end{equation}
		
		By the fact that $ U$ is fast decaying, it follows that for any $\alpha>0$ there exists some $\bar n=\bar{n}(\alpha)\in \N$  such that
		\begin{equation}\label{eq:cond-bar-n}
			U(n)\le e^{-\alpha}\ ,\qquad \forall n\ge \bar n\ .
		\end{equation}
		Moreover, by the fact that $U$ is sub-multiplicative,
		splitting the walk $(X_1,\dots,X_k)$ into smaller walks of length $2\bar n$ {(assuming that~$k$ is a multiple of~$2\bar n$),} we conclude that, a.s.,
		\begin{equation}\label{eq:split}
			\prod_{x \in V}  U(n_x^{(k)}(X))=\prod_{x \in V}  U\left(\sum_{i=1}^{{k/(2\bar n)} }n_x^{(2\bar n)}(X^{(i)})\right)\le \prod_{x \in V}\prod_{i=1}^{{k/(2\bar n)} }  U(n_x^{(2\bar n)}(X^{(i)}))\, .
		\end{equation}
		Therefore, it suffices to notice that the events
		\begin{equation}\label{eq:event-A}
			\cA_{i}=\{\exists x\in V\text{ s.t. }n_x^{(2\bar n)}(X^{(i)})=\bar n \}
		\end{equation}
		are such that the family of Bernoulli random variables $(\mathbbm{1}_{\cA_i})_i$ is i.i.d.\ of parameter $\mathbf{P}_o(\mathcal{A}_1)$. Moreover,
		\begin{equation}
			\mathbf{P}_o(\mathcal{A}_1){\ge\mathbf{P}_o\big(n_o^{(2\bar n)}(X^{(1)}))=\bar n\big)}=(2d)^{-\bar n}\ ,
		\end{equation}
		indeed, the walk has ${2\bar n}$ steps available and has to visit $o$ exactly ${\bar n}$ times, and this can only be done by coming back to $o$ at {every} even step.
		Therefore, by {a} Chernoff bound, {calling}
		\begin{equation}\label{eq:def-R}
			R_k=\bigg|\bigg\{i\le{\frac{k}{2\bar n}}\ :\  \mathbbm{1}_{\cA_i}=1 \bigg\}\bigg|\ ,
		\end{equation}
		as soon as $\varepsilon<p$ we have
		\begin{equation}\label{eq:exp-bound-small-R}
			\mathbf{P}_o\left(R_k \le \varepsilon k \right)\le\exp\left(-\frac{\big(1-\frac{\varepsilon}{p}\big)^2p}{2}\times k\right)\ ,\qquad \forall k\in\N\ ,
		\end{equation}
		where
		$$p=\frac{1}{2\bar n(2d)^{\bar n}}\ . $$
		Moreover, 
		using \eqref{eq:split} and  the fact that $U$ is nice we also have, a.s.,
		\begin{equation}\label{eq:as-ineq}
			\prod_{x \in V}  U(n_x^{(k)}(X))\le \prod_{x \in V}\prod_{i=1}^{{k/(2\bar n)}}  U(n_x^{(2\bar n)}(X^{(i)}))\le \prod_{i=1}^{{k/(2\bar n)}}\prod_{\substack{x\in V\\n_x^{(2\bar n)}(X^{(i)})=\bar n}}  U(n_x^{(2\bar n)}(X^{(i)}))\ .
		\end{equation}
		By \eqref{eq:exp-bound-small-R} and \eqref{eq:as-ineq} we conclude
		\begin{align}
			\mathbf{E}_o\left[\prod_{x \in V}  U(n_x^{(k)}(X)) \right]\le \exp\left(-\frac{\big(1-\frac{\varepsilon}{p}\big)^2p}{2}\times k\right)+e^{-\alpha\frac\varepsilon2 k}\ .
		\end{align}
		At this point, we notice that to prove \eqref{eq:limit-free-energy-bar} it is sufficient to choose  $\varepsilon>0$ which satisfies
		\begin{equation}\label{eq:eps-cond}
			\varepsilon<p\,,\qquad\min\left\{\frac{\big(1-\frac{\varepsilon}{p}\big)^2p}{2}\,,\,\alpha\frac{\varepsilon}{2}\right\}>\log(2d\beta)\,.
		\end{equation}
		Notice that the only value of $\varepsilon<p$ for which the two terms in the minimum in \eqref{eq:eps-cond} coincide is
		\begin{equation}
			\varepsilon=\frac{1}{2} \left(\alpha-\sqrt{\alpha(\alpha+4)}+2\right)p\,,
		\end{equation}
		since $\frac{1}{2} (\alpha-\sqrt{\alpha(\alpha+4)}+2)\in(0,1)$ for all $\alpha>0$.
		In conclusion, for any  choice of $\alpha>0$, $\bar{n}=\bar{n}(\alpha)$ satisfying \eqref{eq:cond-bar-n} and $\varepsilon(\alpha,\bar n)$ as in \eqref{eq:eps-cond}, we have
		\begin{equation}\label{eq:final-quant-est}
			\tilde{\beta}_c(U)\ge \frac1{2d}\exp\left(\frac{\alpha-\sqrt{\alpha(\alpha+4)}+2}{4\bar n (2d)^{\bar n}}\right)\ .
		\end{equation}
	\end{proof}

	\begin{remark}\label{rmk:thm1.1-quant}
		For a given choice of the nice weight function $U$ the estimate in \eqref{eq:final-quant-est} can be optimized over $\alpha$.  Nevertheless, the estimate is not intended to be sharp. Indeed, both the choice of the events in \eqref{eq:event-A} is conceived only to the scope of a general proof, embracing all the possible nice weight functions. Notice further that, for a specific choice of $U$ a quantitative estimate in the spirit of \eqref{eq:final-quant-est} might be interesting in itself since -- as pointed out right before the proof -- it coincides with a quantitative estimate on the connectivity constant of a weakly self-avoiding walk.
	\end{remark}

\section{Random path model}
\label{sect:randompathmodel}
In this section we introduce the Random Path Model (RPM), which is a random loop model that differs from the Random Walk Loop Soup (RWLS) defined in the introduction. The RPM plays an important role in the {proof of Theorem \ref{theo:coexistence}}. In \cite{QuitmannTaggi} it was shown that the RPM and RWLS are \textit{equivalent}. We will first define the RPM and then recall the equivalence relation between the two models.
\subsection{Definition of the RPM} \label{sec:RPMuncoloured}
Let $G=(V,E)$ be a finite undirected %simple, 
graph, and let $N >0$. A realisation of the random path model can be viewed as a collection of 
unoriented paths which might be closed or open.
To define a realisation we need to introduce \textit{links} and 
\textit{pairings}. 
We represent a \textit{link configuration} by
$m \in  \mathcal{M} :=  (\mathbb{N}_0)^E$.
More specifically
$$m = \big ( m_e \big )_{e \in E}\,,$$
where $m_{e}\in\N_0$ represents the number of links on the edge $e$.
Intuitively, a link represents a `visit' at the edge by a path.
The links are ordered  by receiving  a label between $1$ and $m_e$. We denote by $(e,p)$ the $p$-th link at $e \in E$ with $p \in [m_e]$. If a link is on the edge $e = \{x,y\}$, then we  say that \textit{it touches} $x$ and $y$. 

Given a link configuration $m \in \mathcal{M}$, a pairing $\pi_x$ for $m$ at $x \in V$ pairs links touching $x$ in such a way that any link touching $x$ can be paired to at most one other link
	touching $x$, and it is not
	necessarily the case that all links touching $x$ are paired to another link at $x$.
	More formally, $\pi_x$ is a partition of the links touching $x$ into sets of at most two links. If a link touching $x$ is paired at $x$ to no other link touching $x$, then we say that the link is \textit{unpaired at $x$}. We denote by $\mathcal{P}_{x}(m)$ the set of all such pairings for $m$ at $x$ and by $\mathcal{P}(m)$ we denote the set of vectors $\pi=(\pi_x)_{x \in V}$ such that $\pi_x \in \mathcal{P}_{x}(m)$ for all $x \in V$. 

A \textit{configuration}  of the random path model is an element $w = ( m,\pi)$
such that $m \in \mathcal{M}$,
and 
$\pi \in \mathcal{P}(m)$. 
We let $\mathcal{W}$ be the set of  such configurations.
Any configuration $w \in \mathcal{W}$ can be seen as a collection of (open or closed) paths (also called \textit{cycles}), which are unrooted and unoriented. For a formal definition of a path see Section \ref{sect:equivalence}. We let $\alpha(w)$ denote the number of paths in a configuration $w \in \Wcal$. 

With slight abuse of notation, we will also view, $m,\pi,\pi_x: \Wcal \to \N_0$ as functions such that for any $w^\prime=(m^\prime, \pi^\prime) \in \Wcal$, $m(w^\prime)=m^\prime$, $\pi(w^\prime)=\pi^\prime$ and $\pi_x(w^\prime)=\pi_x^\prime$.
For any $x \in V$, let $u_x: \mathcal{W} \rightarrow  \mathbb{N}_0$ be the function corresponding to the number of links touching $x$ which are not paired to any other link touching $x$.
Let further $n_x : \mathcal{W} \rightarrow \mathbb{N}_0$ be the function corresponding to the number of pairings at $x$.

We let 
\begin{equation} \label{eq:randompathtilde}
{\widetilde{\Wcal}}:=\{w \in \Wcal: u_x=0 \ \forall x \in V\}
\end{equation}
be the set of configurations in which there exist no unpaired links.

We now introduce a (non-normalised) measure  on $\mathcal{W}$ and a probability measure on \smash{${\widetilde{\Wcal}}$}.
\begin{definition}\label{def:measure}
	Let $N, \beta \in \R^+$, $U:\N_0 \to \R_0^+$ be given.  We introduce the (non-normalised) non-negative measure $\mu_{G,U,N,\beta}$ on $\mathcal{W}$,
	\begin{equation}\label{eq:RPMmeasure}
		\forall w = (m, \pi) \in \mathcal{W}
		\quad \quad 
		\mu_{G,U,N,\beta}(w) := N^{\alpha(w)} \,
		\bigg(\prod_{e \in {E}} \frac{ \beta^{m_e(w)}}{m_e(w)!} \bigg ) \bigg( \prod_{x \in {V}} U(n_x(w))
		\bigg).
	\end{equation}
	
	Given a function $f : \mathcal{W} \rightarrow \mathbb{C}$, we represent its average by 
	$
	\mu_{G,U,N,\beta} \big ( f \big ) :=
	\sum\limits_{w \in \mathcal{W}}\mu_{G,U,N,\beta} (w) f(w).
	$
	
	We define the measure $\mu^\ell_{G,U,N,\beta}$ as the restriction of the measure $\mu_{G,U,N,\beta}$ to the set of configurations \smash{${\widetilde{\Wcal}}$} and define a probability measure on \smash{${\widetilde{\Wcal}}$} by
	\begin{equation}
		\label{eq:RPMprobabilitymeasure}
		\forall w =(m,\pi) \in {\widetilde{\Wcal}}
		\quad \quad 
		\P_{G,U,N,\beta}(w) := \frac{\mu^\ell_{G,U,N,\beta}(w)}{\Z_{G,U,N,\beta}^\ell},
	\end{equation}
	where \smash{$\Z_{G,U,N,\beta}^\ell := \mu_{G,U,N,\beta}({\widetilde{\Wcal}})$}  is the partition function.
	We denote by $\E_{G,U,N,\beta}$ the expectation under the measure $\P_{G,U,N,\beta}$. Sometimes, for a lighter notation, we will omit the sub-scripts. 
\end{definition}

\subsection{Equivalence} \label{sect:equivalence}
In \cite[Section 3]{QuitmannTaggi} it is shown that the RWLS and the RPM are equivalent if one considers the expectation of functions which do not depend on certain features of the configurations. In \cite{QuitmannTaggi} the RPM was introduced slightly differently, namely, each path in a configuration has a colour from $1$ to $N$, and the factor $N^{\# \text{ loops}}$ does not appear in the definition of the measure. The equivalence between the RPM and RWLS was thus shown for any $N \in \N$ only, but with the definition of the RPM in Section \ref{sec:RPMuncoloured}, the result of \cite[Section 3]{QuitmannTaggi} can easily be extended to any $N>0$. 

In this section we will present a special case of \cite[Theorem 3.14]{QuitmannTaggi} that we will apply in the proof of Theorem~\ref{theo:coexistence} and that further reflects the similarity between the two models. Informally said, we will show that the expected number of loops of any given shape is identical in both models. 
We will first introduce an equivalence relation on $\Lcal$ and then define {so-called} multiplicity numbers. 

Throughout this section, we fix an arbitrary undirected, finite graph $G=(V,E)$, parameters $N, \beta \in \R^+$ and a {fast decaying weight} function $U:\N_0 \to \R_0^+$.

\paragraph{Equivalence classes of rooted oriented loops in the RWLS.}
We now introduce an equivalence relation on $\Lcal$. 
We call two r-o-loops $\ell,\ell^\prime \in \Lcal$ \textit{equivalent} if one sequence can be obtained {as} a time-reversion,  a cyclic permutation or a combined time-reversion and cyclic permutation of the other sequence, see \cite[Section~3]{QuitmannTaggi} for a precise definition.
We denote the equivalence class of $\ell\in \Lcal$ by $\gamma(\ell)$ and by $\Sigma(\Lcal)$ we denote the set of equivalence classes of r-o-loops.

\paragraph{Multiplicity numbers in the RWLS.} For any $\gamma \in \Sigma(\Lcal)$ we introduce the \textit{multiplicity number} $k_\gamma: \Omega \to \N_0$, 
which for any $\omega=(\ell_1, \ldots, \ell_{|\omega|}) \in \Omega$ counts the number of r-o-loops in $\omega$ that are an element of $\gamma$, i.e., 
\begin{equation} \label{eq:multiplicityRWLS}
	k_{\gamma}(\omega) := \Big | \big \{   i \in \big\{1, \ldots, |\omega| \big\} \, : \, \ell_i \in \gamma  \big \}  \Big |\,.
\end{equation}
For~$e=\{x,y\}\in E$ and $\gamma=\gamma((x,y,x)) \in \Sigma(\Lcal)$, we also write $k_e=k_\gamma$ for the number of loops of length two on the edge~$e$.

\paragraph{Multiplicity numbers in the RPM.}
Given $m \in {\Mcal}$, a \textit{rooted oriented linked loop} (in short: \textit{r-o-l-loop}) for $m$ is an ordered sequence of nearest-neighbour vertices and pairwise distinct links,
$$L=(x_0,(\{x_0,x_1\},p_1),x_1,(\{x_1,x_2\},p_2),\dots,(\{x_{k-1},x_{k}\},p_{k}),x_{k}\big)\,,
$$ where $p_j \in \{1,\dots,m_{\{x_{j-1},x_j\}}\}$ for each $j \in [k]$ and such that $x_0=x_{k}$ for some ${k\ge 2}$. 
 The vertex $x_0$ is called the \textit{root} of $L$, the link $(\{x_0,x_1\},p_1)$ is called the \textit{starting link} of $L$ 
and the length of $L$ is defined by $|L|:=k$. 
We let $\mathcal{R}_m$ be the set of all r-o-l-loops for $m$. Two r-o-l-loops are said to be \textit{equivalent} if one sequence can be obtained from the other sequence as a time-inversion, or as a cyclic permutation, or a combination of the two; see again \cite[Section 3]{QuitmannTaggi}.
We denote by $\chi(L) \subset \mathcal{R}_m$ the equivalence class of $L \in \mathcal{R}_m$ and we denote by $\Sigma(\Rcal_m)$ the set of %all 
equivalence classes of r-o-l-loops. The equivalence class of a r-o-l-loop will simply be referred to as \textit{cycle}, which can be thought of as an unoriented loop with no root. 

A set of cycles for $m$, {$A\subset\Sigma(\Rcal_m)$} is called an \textit{ensemble of cycles} for $m$ if every link is contained in precisely one cycle of the set, i.e., if for every $(e,p)$ with {$e\in E$ and~$p\in[m_e]$}, there exists precisely one {$\chi\in A$} such that $(e,p) \in {\chi}$. We denote by $\Ecal_m$ the set of ensembles of cycles for $m$.

For any \smash{$w=(m,\pi) \in {\widetilde{\Wcal}}$}, we can uniquely construct an ensemble of cycles 
\begin{equation} \label{eq:ensemblecycles}
	\zeta(w):=\{\chi_1(w),\dots,\chi_{k(w)}(w)\} \in \Ecal_m
\end{equation}
as follows: We take any link $(\{x,y\},p_{\{x,y\}})$ and choose a point $z_0 \in \{x,y\}$. {Step by step}, we construct a r-o-l-loop $L_1$ for $m$ with root $z_0$ and with starting link $(\{x,y\},p_{\{x,y\}})$ by choosing $z_1 \in \{x,y\} \setminus \{z_0\}$ as the next vertex and by choosing the link which is paired at $z_1$ to $(\{x,y\},p_{\{x,y\}})$ as the next link in the sequence. We continue until we are back at the link $(\{x,y\},p_{\{x,y\}})$. We define the cycle $\chi_1(w) \in \Sigma(\Rcal_m)$ as the equivalence class of $L_1$. For the next cycle, we choose a link that has not been selected yet and proceed as before. We continue until all links have been selected. 

For $m \in {\Mcal}$, we introduce the map 
\begin{equation}\label{eq:def-vartheta}
	\vartheta: \Rcal_m \to\Lcal\,,
\end{equation}
which acts by projecting a r-o-l-loop $L=(x_0,(\{x_0,x_1\},p_1),x_1,\dots,(\{x_{k-1},x_{k}\},p_{k}),x_{k}\big) \in \Rcal_m$ onto the corresponding r-o-loop $\vartheta(L):=(x_0,x_1,\dots,x_k) \in {\Lcal}$ by \lq forgetting\rq{} about the links in the sequence. It is important to note that {any two r-o-l-loops~$L,L^\prime \in \Rcal_m$ are equivalent if and only if the r-o-loops $\vartheta(L), \vartheta(L^\prime) \in {\Lcal}$ are equivalent}. By slight abuse of notation, we thus also use the function $\vartheta$ to map an equivalence class $\chi \in \Sigma(\Rcal_m)$ to its corresponding equivalence class $\vartheta(\chi) \in \Sigma({\Lcal})$.

	For any $\gamma \in \Sigma({\Lcal})$ we introduce the \textit{multiplicity number} \smash{$\tilde{k}_\gamma: {\widetilde{\Wcal}} \to \N_0$}, 
	which for any \smash{$w \in {\widetilde{\Wcal}}$} counts the number of  times a cycle $\chi \in \zeta(w)$ projects onto $\gamma$, i.e., 
	\begin{equation}\label{eq:definitionmultiplicitynumberedge2}
		\tilde{k}_\gamma(w) := \Big | \big \{   {\chi\in\zeta(w)} \, : \, \vartheta({\chi})= \gamma  \big \}  \Big |\,.
	\end{equation}
	For~$e=\{x,y\}\in E$ and $\gamma=\gamma((x,y,x)) \in \Sigma(\Lcal)$, we also write $\tilde{k}_e=\tilde{k}_\gamma$.

	A special case of \cite[Theorem 3.14]{QuitmannTaggi} is the following lemma, which states that the expected multiplicity numbers are identical in the RWLS and RPM. {We recall from \eqref{eq:measure} and \eqref{eq:RPMmeasure} that in \eqref{eq:expectationequiv1} and \eqref{eq:expectationequiv2} below the expectation in the right and left hand-side refer to the RPM and RWLS, respectively.} 
	
	\begin{lemma} \label{lemma:equivalence}
		For any $A \subset E$ and any functions $f_e:\N_0 \to \R$, $e \in A$, 
		\begin{equation} \label{eq:expectationequiv1}
			\Ecal_{G,U,N,\beta}\bigg(\prod_{e \in A} f_e(k_e)\bigg)=\E_{G,U,N,\beta}\bigg(\prod_{e \in A} f_e(\tilde{k}_e)\bigg)\,.
		\end{equation}
		Further, for any $\gamma \in \Sigma(\Lcal)$,
		\begin{equation} \label{eq:expectationequiv2}
			\Ecal_{G,U,N,\beta}(k_\gamma)= \E_{G,U,N,\beta}(\tilde{k}_\gamma)\,.
		\end{equation}
	\end{lemma}

\section{Proof of Theorem \ref{theo:coexistence}}
\label{sect:coexistence}
In this section we prove Theorem \ref{theo:coexistence}. We present the proofs of the upper and lower {bounds} separately in Sections \ref{sect:upperbound} and \ref{sect:lowerbound}, respectively.

\subsection{Uniform upper bound on the density of any fixed microscopic loop} \label{sect:upperbound}
Throughout this section, we fix an arbitrary finite undirected bipartite graph $G=(V,E)$ with $V=V_1\sqcup V_2$ and each edge in $E$ connects a point of~$V_1$ with a point of~$V_2$. Further, we fix parameters $N,\beta \in \R^+$ and a fast decaying weight function $U:\N_0 \to {\R_0^+}$. The proof of the upper bound in Theorem~\ref{theo:coexistence} will follow from Proposition \ref{prop:coexistence} below. Before stating the proposition, we need to introduce some functions acting on $\Sigma(\Lcal)$, where we recall from Section \ref{sect:equivalence} that $\Sigma(\Lcal)$ denotes the set of {equivalence} classes of $\Lcal$.

Given two r-o-loops $\ell,\ell^\prime \in \Lcal$ such that $\ell(0)=\ell^\prime(0)$, we define their concatenation  as
$\ell \oplus \ell^\prime:=\big(\ell(0),\dots,\ell(|\ell|),$ $\ell^\prime(1),\dots, \ell^\prime(|\ell^\prime|)\big).$ 
We define the \textit{multiplicity}  of $\ell$, $J(\ell)$, as the maximal integer $n \in \N$ such that $\ell$ can be written as the $n$-fold concatenation of some r-o-loop, $\tilde{\ell}$,  with itself. 
%Such a  loop $\tilde{\ell} = \tilde{\ell}(\ell)$ has multiplicity one and will be referred to as the \textit{elementary loop} of $\ell$.
{We also define the \textit{time-reversal} of~$\ell$ as~$r(\ell) := \big(\ell(|\ell|),\,\ell(|\ell|-1),\,\ldots,\,\ell(0)\big)$.}
We call a r-o-loop $\ell$ \textit{stretched} if there exists a cyclic permutation of $\ell$ that is identical to {its time reversal}~$r(\ell)$.
%, i.e., if there exists $m \in \{0,\dots,|\ell|-1\}$ such that $c_m(\ell)=r(\ell)$. 
Otherwise the r-o-loop is called \textit{unstretched}, see also \cite[Figure~3]{QuitmannTaggi}.
For any $\ell \in {\Lcal}$, we define the \textit{stretch-factor} $\delta(\ell)$ by
$$
\delta(\ell):=\begin{cases} 1  & \text{if } \ell \text{ is stretched,} \\
	2 & \text{if } \ell \text{ is unstretched.} \end{cases}
$$
%Note that, for any pair of equivalent r-o-loops $\ell,\ell^\prime \in \Lcal$, it holds that $J(\ell)=J(\ell^\prime)$ and $\delta(\ell)=\delta(\ell^\prime)$.  Thus, by slight abuse of notation, we also use the notations $\delta(\gamma)$ and  $J(\gamma)$ for the equivalence classes $\gamma \in \Sigma(\Lcal)$. Further, we denote by $\alpha(\gamma)$ the length of any r-o-loop in $\gamma$ and by $|\gamma|$ we denote the cardinality of $\gamma$.
%
%
%For every rooted oriented loop $\ell=\big(\ell(0),\,\ell(1),\,\ldots,\,\ell(2k)\big) \in \Lcal$, and for every edge~$e\in E$, we set
%$$
%q_e(\ell)
%\ :=\ 
%\big|
%\big\{
%j\in\{0,\,\ldots,\,k-1\}\ :\ 
%\{\ell(2j),\,\ell(2j+1)\}=e
%\big\}\big|,
%$$
%which is the number of even steps of $\ell$ on $e$.
%For any $\gamma\in\Sigma(\mathcal{L})$, we choose an arbitrary loop~$\ell\in\gamma$ and we define~$q_e(\gamma)=q_e(\ell)$ for every edge~$e\in E$ (the result depends on the choice of the loop, which amounts to choosing which steps are odd and which steps are even in~$\gamma$).
%For every~$x\in V$, we denote by $n_x(\gamma)$ the number of visits at~$x$ by $\gamma$, namely,
%$$
%n_x(\gamma)
%\ :=\ 
%\sum_{y\sim x}q_{\{x,y\}}(\gamma)
%\,,
%$$
%and we define the support of the loop~$\gamma$ by
%$$
%\mathrm{supp}(\gamma)
%\ :=\ 
%\big\{x\in V\ :\ n_x(\gamma)>0\big\}
%\,.
%$$

For any $x\in V$ and $\ell \in \Lcal$, we denote by $n_x(\ell)$ the number of visits at~$x$ by $\ell$, namely,
$$
n_x(\ell):=\sum_{i=0}^{|\ell|-1} \mathbbm{1}_{\{\ell(i)=x\}}\,.
$$
We define the support of $\ell$ by
\begin{equation}\label{eq:def-supp}
	\mathrm{supp}\,(\ell)
	\ :=\ 
	\big\{x\in V\ :\ n_x(\ell)>0\big\}
	\,.
\end{equation}
Note that, for any pair of equivalent r-o-loops $\ell,\ell^\prime \in \Lcal$, {we have} $J(\ell)=J(\ell^\prime)$, $\delta(\ell)=\delta(\ell^\prime)$, $n_x(\ell)=n_x(\ell^\prime)$ and $\mathrm{supp}\,(\ell)= \mathrm{supp}\,(\ell^\prime)$.  Thus, by slight abuse of notation, we also use the notations $J(\gamma), \delta(\gamma), n_x(\gamma)$ and $\mathrm{supp}(\gamma)$ for  equivalence classes $\gamma \in \Sigma(\Lcal)$. Further, we denote by $\alpha(\gamma)$ the length of any r-o-loop in $\gamma$ and by $|\gamma|$ we denote the cardinality of the class $\gamma$.

For any $a \in \N$ and $\gamma \in \Sigma(\Lcal)$, we further define
\begin{equation}
	\mu_a(\gamma)
	\ :=\ 
	\mathcal{E}_{G, U,  N, \beta}
	\Big[\,
	k_\gamma(k_\gamma-1)\ldots(k_\gamma-a+1)
	\,\Big]\,,
\end{equation}
where we recall from \eqref{eq:multiplicityRWLS} that $k_\gamma(\omega)$ counts the number of loops in $\omega \in \Omega$ that are an element of the equivalence class $\gamma$. 
We now have all ingredients to state Proposition \ref{prop:coexistence}.

\begin{proposition}
	\label{prop:coexistence}
	For any $a \in \N$ and $\gamma\in\Sigma(\mathcal{L})$, 
	\begin{equation}
		\label{eq:moment_bound}
		\mu_a(\gamma)
		\ \leq\ 
		\lambda(\gamma)^a, 
	\end{equation}
	where
	$$
	\lambda(\gamma)
	\ :=\ 
	\frac{\delta(\gamma)}{J(\gamma)}
	\times\frac{N}{2}
	\times \max\bigg\{e^{\frac{N}{2}},\, \frac{2e}{N} \bigg\}^{\frac{\alpha(\gamma)}{2}}\,.
	$$
\end{proposition}

%Note that, through Chebychev's inequality, the upper bound~\eqref{eq:moment_bound} implies that~$k_\gamma$ has Poisson-like tails, namely that, for all~$k\in\mathbb{N}$ with~$k\geq\lambda(\gamma)$,
%$$
%	\mathcal{P}_{G, U,  N, \beta}
%	\big(k_\gamma\,\geq\,k\big)
%	\ \leq\ 
%	\frac{\mu_{k-\lfloor\lambda(\gamma)\rfloor}(\gamma)}{k(k-1)\ldots\big(\lfloor\lambda(\gamma)\rfloor+1\big)}
%	\ \leq\ 
%	\frac{\lfloor\lambda(\gamma)\rfloor!}{\lambda(\gamma)^{\lfloor\lambda(\gamma)\rfloor}}
%	\times
%	\frac{\lambda(\gamma)^k}{k!}
%	\,.
%	$$
We remark that \eqref{eq:moment_bound} implies that $k_\gamma$ has Poisson-like tails. More precisely, applying Markov's inequality, we deduce from \eqref{eq:moment_bound} that for any~$a\in\mathbb{N}$ with~$a\geq\lambda(\gamma)$,
$$
\Pcal_{G, U,  N, \beta}\big(k_\gamma \geq a \big) 
= \Pcal_{G, U,  N, \beta}\big(k_\gamma(k_\gamma-1)\cdots (k_\gamma-a+1) \geq a!\big) 
\leq \frac{\mu_a(\gamma)}{a!} \leq \frac{\lambda(\gamma)^a}{a!}.
$$

We now explain how the upper bound of Theorem~\ref{theo:coexistence} follows from Proposition~\ref{prop:coexistence}.

\begin{proof}[Proof of the upper bound in Theorem~\ref{theo:coexistence}]
Let $k \in \N$ and $\Delta$ denote the maximal degree of $G$.
 Applying Proposition \ref{prop:coexistence}  with $a=1$ and bounding $\delta\le 2$ and $J\ge 1$, we deduce that,
%\begin{equation} \label{eq:proofupperbound}
%\rho_K(2k)
%\ =\ 
%\frac{1}{|K|} \sum_{m=1}^{k}
%\sum_{\substack{\gamma\in\Sigma(\mathcal{L}):\\  \alpha(\gamma)=2m}}
%\mathcal{E}_{G_K, U,  N, \beta}
%\big[k_\gamma\big]
%\ \leq\ 
%\frac{4}{N}
%\sum_{m=1}^{k}
%\Bigg[
%\,
%\delta^2
%\left(1\vee\frac{2}{N}\right)\exp\left(1\vee\frac{N}{2}\right)
%\,\Bigg]^m
%\,,
%\end{equation}
\begin{equation} \label{eq:proofupperbound}
	\rho(2k)
	\ =\ 
	\frac{1}{|V|} 
	\sum_{\substack{\gamma\in\Sigma(\mathcal{L}):\\  \alpha(\gamma)=2k}}
	\mathcal{E}_{G, U,  N, \beta}
	\big[k_\gamma\big]
	\ \leq\ 
{N}
	\Bigg[
	\,
	\Delta^2 \, \max\bigg\{e^{\frac{N}{2}}, \frac{2e}{N} \bigg\}
	\,\Bigg]^k
	\,,
\end{equation}
where we used that $|\{\gamma \in \Sigma(\Lcal): \, \alpha(\gamma)=2k\}| \leq |V| \, \Delta^{2k}$ for any $k \in \N$. 
Since \eqref{eq:proofupperbound} holds {uniformly} with respect to~$\beta$ and $|V|$, this finishes the proof of the upper bound in Theorem \ref{theo:coexistence}.
\end{proof}

The remainder of the current section is devoted to the proof of Proposition \ref{prop:coexistence}. 
The proof consists of two main steps. In Section \ref{suse:decomposition} we  first derive a formula for $\mu_a(\gamma)$ in terms of loops of length two, which reduces the problem to the study of the distribution of such short loops. The formula is written in Lemma \ref{lemma:colourswitch} below. In Section \ref{suse:ewens} we then derive a control on the distribution of  \textit{double links}, i.e., loops of length two, in the RPM. As a preliminary step, in Lemma \ref{lemma:ewens}, we show a conditional independence property of the RPM which provides a connection with Ewens Permutations. The stochastic upper bound on the density of double links is then presented in Lemma \ref{lemma:upperstochdomination} and follows from an elementary uniform upper bound on the number of fixed points in a large random permutation. In Section \ref{suse:conclusion-ub}, we apply Lemma \ref{lemma:colourswitch} and Lemma \ref{lemma:upperstochdomination}, as well as the equivalence between the RPM and RWLS (see Lemma \ref{lemma:equivalence}) to conclude the proof of Proposition \ref{prop:coexistence}.

\subsubsection{Decomposition of loops into loops of length two}\label{suse:decomposition}

In this section we will decompose any loop $\ell \in \Lcal$ into loops of length two. For this purpose we introduce the function $q_e(\ell)$, $\ell \in \Lcal$, $e\in E$, which counts the number of even steps of the loop $\ell$ on the edge $e$, namely,
$$
q_e(\ell)
\ :=\ 
\big|
\big\{
j\in\{0,\,\ldots,\,|\ell|-1\}\ :\ 
\{\ell(2j),\,\ell(2j+1)\}=e
\big\}\big|\,.
$$
For any $\gamma\in\Sigma(\mathcal{L})$, we choose an arbitrary loop~$\ell\in\gamma$ and we define~$q_e(\gamma)=q_e(\ell)$ for any $e\in E$ (the result depends on the choice of the loop, which amounts to choosing which steps are odd and which steps are even in~$\gamma$).
\begin{lemma}
	\label{lemma:colourswitch}
For any~$k\in\mathbb{N}$, any $\gamma\in\Sigma(\mathcal{L})$ of length~$\alpha(\gamma)=2k$ and for any~$a\in\mathbb{N}$, we have that,
\begin{equation} \label{eq:loopdecomp}
	\mu_a(\gamma)
	\ =\ 
	\psi(\gamma)^a
	\times
	\mathcal{E}_{G, U,  N, \beta}
	\left[\,
	\prod_{e\in E}
	\,
	\prod_{0\,\leq\,i\,<\,a\times q_e(\gamma)}
	(k_e-i)
	\,\right],
	\qquad
	\text{where}
	\quad
	\psi(\gamma)
	\ =\ 
	\frac{\delta(\gamma)}{J(\gamma)}\left(\frac{2}{N}\right)^{k-1}
	\,,
\end{equation}
	with the natural convention that the second product is~$1$ when~$q_e(\gamma)=0$.
\end{lemma}

\begin{proof}
Let {$k \in \N$ and $\gamma \in \Sigma(\Lcal)$} with $\alpha(\gamma)=2k$ and let $a \in \N$. If $k=1$, then \eqref{eq:loopdecomp} is trivial since $\delta(\gamma)=J(\gamma)=1$. Assume now that $k \geq 2$, i.e., that $\gamma$ is not a loop of length two.
	Let us consider the event
	$$\Omega_{{\gamma,}a}
	\ =\ 
	\big\{\omega\in\Omega\ :\ k_{\gamma}(\omega)\geq a\big\}\,,$$
	which is the set of configurations in which at least~$a$ occurrences of loops belonging to the equivalence class~$\gamma$ are present.
	
	On this set we introduce the map~$\phi:\Omega_{{\gamma,}a}\to\Omega$, which for any $\omega \in \Omega_{{\gamma,}a}$ acts by removing the first~$a$ loops which belong to the equivalence class~$\gamma$, by reindexing the remaining loops without changing their order, and by adding~$a\times q_e(\gamma)$ loops of length two on each edge~$e\in E$ (choosing arbitrarily the {root} of each loop), appending these loops at the end of the sequence that is obtained from $\omega$ after the removal of the $a$ loops, with an arbitrary deterministic order on~$E$, {see also Figure \ref{fig:loopdecomp}}.

\begin{figure}
\centering
\begin{subfigure}{.3\textwidth}
  \centering
  \includegraphics[width=\textwidth]{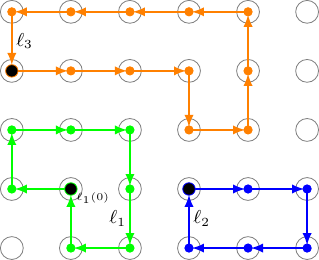}
  \caption{}
  \label{fig:sub1}
\end{subfigure}
\hspace{6em}
\begin{subfigure}{.3\textwidth}
  \centering
  \includegraphics[width=\textwidth]{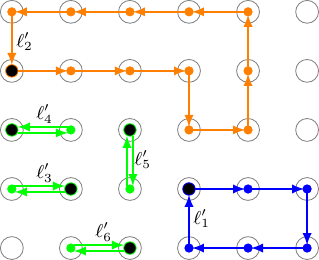}
  \caption{}
  \label{fig:sub2}
\end{subfigure}
\caption{(a) A configuration $\omega=(\ell_1,\ell_2,\ell_3) \in \Omega_{\gamma,1}$, where $\gamma:=\gamma(\ell_1) \in \Sigma(\Lcal)$. The roots of the rooted oriented loops in $\omega$ are represented by the larger filled circles. (b) The configuration $\phi(\omega)=(\ell^\prime_1,\ell^\prime_2,\ell^\prime_3,\ell^\prime_4, \ell_5^\prime, \ell_6^\prime) \in \Omega$, where we set $q_e(\gamma)=q_e(\ell_1)$ for any $e \in E$.}
\label{fig:loopdecomp}
\end{figure}

	Let us now define an equivalence relation for RWLS configurations, which is taken from \cite{QuitmannTaggi}.
	We call two configurations~$\omega,\,\omega'\in\Omega$ \textit{equivalent} if there exists~$n\in\mathbb{N}$ such that~$\omega=(\ell_1,\,\ldots,\,\ell_n)$ and~$\omega'=(\ell'_1,\,\ldots,\,\ell'_n)$ and if there exists a permutation~$\pi\in S_n$ such that~$\ell_{\pi(i)}\in\gamma(\ell'_i)$ for all~$i\in[n]$.
	We denote by~$\Sigma(\Omega)$ the set of equivalence classes of~$\Omega$, and by~$\Sigma(\Omega_{{\gamma,}a})$ the set of equivalence classes of configurations in~$\Omega_{{\gamma,}a}$.
	
	Note that, the functions~$k_{\gamma'}$ for~$\gamma'\in\Sigma(\Lcal)$, as well as the local time, are constant on equivalence classes, which allows us to define~$k_{\gamma'}(\rho)$,~$k_e(\rho)$ and~$n_x(\rho)$ for an equivalence class~$\rho\in\Sigma(\Omega)$.
For any~$\rho\in\Sigma(\Omega)$, we set
	$$\nu(\rho)
	\ =\ 
	\sum_{\omega\in\rho}
	\mathcal{P}_{G, U,  N, \beta}(\omega)\,. $$
	By \cite[Eq.~(3.19)]{QuitmannTaggi} we have that,
	\begin{equation}
		\label{eq:nu_rho}
		\nu(\rho)
		\ =\ 
		\frac{1}{\mathcal{Z}_{G,  U,  N,  \beta}}
		\prod_{\gamma'\in\Sigma(\mathcal{L})}
		\frac{1}{k_{\gamma'}(\rho)!}
		\left(\frac{N\beta^{\alpha(\gamma')}\delta(\gamma')}{2J(\gamma')}\right)^{k_{\gamma'}(\rho)}
		\prod_{x\in V}
		U\big(n_x(\rho)\big)\,,
	\end{equation}
	where the product is in fact finite.
	
	For all configurations~$\omega,\,\omega'\in\Omega_{{\gamma,}a}$, we have that~$\omega$ and~$\omega'$ are equivalent if and only if~$\phi(\omega)$ and~$\phi(\omega')$ are equivalent.
	Therefore, we may define a lifted map~$\Phi:\Sigma(\Omega_{{\gamma,}a})\to\Sigma(\Omega)$ which is such that~$\Phi(\rho(\omega))=\rho(\phi(\omega))$ for every configuration~$\omega\in\Omega_{{\gamma,}a}$, and which is injective.
	We then have
	\begin{equation}
	\label{eq:imPhi}
	\Phi\big(\Sigma(\Omega_{{\gamma,}a})\big)
	\ =\ 
	\Big\{\rho\in\Sigma(\Omega)\ :\ 
	\forall e\in E,\ 
	k_e(\rho)\,\geq\,a\times q_e(\gamma)
	\Big\}
	\,.
	\end{equation}
	
	We now wish to compare~$\nu(\rho)$ with~$\nu(\Phi(\rho))$.
	On the one hand, note that our map~$\phi$, and hence also~$\Phi$, does not change the local time at any vertex of the graph.
On the other hand, recalling our assumption that~$\gamma$ is not the class of a loop of length two,  we have that, for any $\rho\in\Sigma(\Omega_{{\gamma,}a})$ and any $\gamma^\prime \in \Sigma(\Lcal)$ with $\alpha(\gamma^\prime)>2$,
\begin{align} \label{eq:changemultiplicitynumberlongloop}
k_{\gamma^\prime}\big(\Phi(\rho)\big) & = 
\begin{cases} 
k_\gamma(\rho)-a & \text{ if } \gamma^\prime=\gamma\,, \\
k_{\gamma^\prime}(\rho) & \text{ if } \gamma^\prime \neq \gamma\,,
\end{cases} \\
\intertext{while, for any $e \in E$, } 
\label{eq:changemultiplicitynumbershortloop}
k_e\big(\Phi(\rho)\big) & = k_e(\rho) +a\times q_e(\gamma)\,.
\end{align}
From (\ref{eq:nu_rho}), \eqref{eq:changemultiplicitynumberlongloop}, \eqref{eq:changemultiplicitynumbershortloop} and using that~$\delta(\gamma')=J(\gamma')=1$ when~$\gamma'$ is the equivalence class of a loop of length two, we obtain that, for every~$\rho\in\Sigma(\Omega_{{\gamma,}a})$,
	\begin{align*}
		\nu\big(\Phi(\rho)\big)
		&\ =\ 
		\nu(\rho)
		\,
		\prod_{\gamma'\in\Sigma(\mathcal{L})}
		\frac{k_{\gamma'}(\rho)!}{k_{\gamma'}\big(\Phi(\rho)\big)!}
		\left(\frac{N\beta^{\alpha(\gamma')}\delta(\gamma')}{2J(\gamma')}\right)^{k_{\gamma'}(\Phi(\rho))-k_{\gamma'}(\rho)}
		\\
		&\ =\ 
		\nu(\rho)
		\,
		\frac{\big(k_{\gamma}(\rho)\big)!}{\big(k_{\gamma}(\rho)-a\big)!}
		\left(
		\frac{2J(\gamma)}{N\beta^{2k}\delta(\gamma)}
		\right)^a
		\prod_{e\in E}
		\frac{k_e(\rho)!}{\big(k_e(\rho)
			+a\times q_e(\gamma)\big)!}
		\left(\frac{N\beta^2}{2}\right)^{a\times q_e(\gamma)}
		\\
		&\ =\ 
		\nu(\rho)
		\,
		\frac{\big(k_{\gamma}(\rho)\big)!}{\big(k_{\gamma}(\rho)-a\big)!}
		\,\psi(\gamma)^{-a}
		\,\prod_{e\in E}
		\frac{k_e(\rho)!}{\big(k_e(\rho)
			+a\times q_e(\gamma)\big)!}
		\,,
		\numberthis\label{eq:nu_phi_rho}
	\end{align*} 
	where, in the last equality, we used the relation
	$$\sum_{e\in E}
	q_e(\gamma)
	\ =\ \frac{\alpha(\gamma)}{2}
	\ =\ k\,.  $$
Using~\eqref{eq:nu_phi_rho}, \eqref{eq:changemultiplicitynumbershortloop} and recalling that~$\Phi$ is injective, we obtain
	\begin{align*}
		\mu_a(\gamma)
		&\ =\ 
		\sum_{\rho\in\Sigma(\Omega_{{\gamma,}a})}
		\nu(\rho)
		\,
		\frac{\big(k_{\gamma}(\rho)\big)!}{\big(k_{\gamma}(\rho)-a\big)!}
		\\
		&\ =\ 
		\sum_{\rho\in\Sigma(\Omega_{{\gamma,}a})}
		\nu\big(\Phi(\rho)\big)
		\,\psi(\gamma)^a
		\,
		\prod_{e\in E}
		\frac{\big(k_e(\rho)
			+a\times q_e(\gamma)\big)!}{k_e(\rho)!}
		\\
		&\ =\ 
		\psi(\gamma)^a
		\sum_{\rho'\in\Phi(\Sigma(\Omega_{{\gamma,}a}))}
		\nu(\rho')
		\,
		\prod_{e\in E}
		\frac{k_e(\rho')!}{\big(k_e(\rho')
			-a\times q_e(\gamma)\big)!}
		\\
		&\ =\ 
		\psi(\gamma)^a
		\sum_{\rho'\in\Sigma(\Omega)}
		\nu(\rho')
		\,
		\prod_{e\in E}
	\,
	\prod_{0\,\leq\,i\,<\,a\times q_e(\gamma)}
	\big(k_e(\rho')-i\big)	
		\,,
	\end{align*}
	where, in the last step, we used the fact that the product is zero when~$\rho'\in\Sigma(\Omega)\setminus\Phi\big(\Sigma(\Omega_{{\gamma,}a})\big)$, which follows from~\eqref{eq:imPhi}.
	This concludes the proof of Lemma~\ref{lemma:colourswitch}.
\end{proof}

\subsubsection*{Decomposition of an open path to study correlation functions}
{As a side} remark, let us point out that the above technique may also be applied to study the correlation functions, by replacing one open path with loops of length two.
More precisely, we have the following result, where we recall from the beginning of this section that we fix a bipartite graph $G=(V,E)$ and denote by $V_1,V_2 \subset V$ the disjoint subsets such that $V=V_1\sqcup V_2$.

\begin{lemma}
	\label{lemma:colourswitch_bis}
%	Let~$G=(V,\,E)$ be a finite undirected graph, let~$U : \mathbb{N}_0 \rightarrow \mathbb{R}^+$ be a fast decaying weight function and let~$\beta,\,N\in\mathbb{R}_0^+$.
%	Assume that~$G$ is bipartite, with~$V=V_1\sqcup V_2$ (and each edge connects a point of~$V_1$ with a point of~$V_2$).
	For every~$x\in V_1$ and~$y\in V_2$, the two-point function defined in~\eqref{eq:spinconnection} 
	 may be expressed as
	$$
	 \frac{\mathcal{Z}_{G, U, N, \beta}(x,y)}{\mathcal{Z}_{G, U, N, \beta}}
	\ =\ 
	\frac{1}{\beta}
	\sum_{\gamma \in\mathcal{L}(x,\,y)}
	\left(\frac{2}{N}\right)^{\frac{|\gamma|+1}{2}}
	\mathcal{E}_{G, U,  N, \beta}
	\left[\,
	\prod_{e\in E}
	\,
	\prod_{0\,\leq\,i\,<\,q_e(\gamma)}
	(k_e-i)
	\,\right]
	\,,
	$$
	where, for every path~$\gamma\in\mathcal{L}_{x,y}$ and every edge~$e\in E$,~$q_e(\gamma)$ is the number of odd steps of the path~$\gamma$ on the edge~$e$.
\end{lemma}

The assumption that $G$ is bipartite is intended to guarantee that all the open paths from~$x\in V_1$ to~$y\in V_2$ have odd length, which has the convenient consequence that the path may be replaced with loops of length two without changing the local time anywhere.

The other difference with Lemma~\ref{lemma:colourswitch} is that there is no multiplicity factor~$\delta(\gamma)/J(\gamma)$ due to the number of possibilities to obtain the loop, because an open path from~$x$ to~$y$ already has a well defined starting point and orientation.

We omit the detailed proof because, apart from these two details, it is very similar to the proof of Lemma~\ref{lemma:colourswitch} and in any case we do not use this result in what follows.

\subsubsection{Ewens Permutations and conditional independence on the RPM}\label{suse:ewens}
Let us denote by~$S_n$ the set of permutations of the set~$[n]=\{1,\,\ldots,\,n\}$.
For every~$\sigma\in S_n$, we denote by~$c(\sigma)$ the number of cycles in the decomposition of~$\sigma$ into disjoint cycles, counting also the cycles of length one as cycles, so that for example~$c(\mathrm{Id})=n$.
The Ewens distribution on the set~$S_n$ is the probability distribution defined by,
\begin{equation}
	\label{defEwens}
	P_{\theta,\,n}^{\text{Ewens}}(\sigma)
	\ =\ \frac{\theta^{c(\sigma)}}{Z(\theta,\,n)}\,,
\end{equation}
where~$\theta>0$ is a parameter and~$Z(\theta,\,n)$ is the normalizing constant.
Following~\cite{Ewens}, we have that
\begin{equation}
	\label{eq:zewens}
	Z(\theta,\,n)
	\ =\ 
	\theta(\theta+1)\ldots(\theta+n-1)
	\,.
\end{equation}
We write~${\rm FP}(\sigma)=\{i\in[n]\,:\,\sigma(i)=i\}$ for the set of the fixed points of the permutation~$\sigma$.

In the next lemma, we study the occurrences of loops of length two in configurations of the RPM, which was defined in Section \ref{sect:randompathmodel}.

For any $e \in E$ and any $w \in \widetilde\Wcal$, we recall from  \eqref{eq:definitionmultiplicitynumberedge2} that $\tilde{k}_e(w)$ denotes the number of double links on $e$, where a double link is defined as a pair of links on $e$ that are paired together at both its endpoints.

For any $\{x,\,y\}\in E$, $m \in \Mcal$ and $\pi_y \in \Pcal_y(m)$, we call a pair of links $(\{x,y\},p),(\{x,y\},p^\prime)$ with $p, p^\prime \in [m_{\{x,y\}}]$ a \textit{vertical pairing on}~$y$ \textit{towards}~$x$ if they are paired together at $y$, i.e., if $\big\{(\{x,y\},p),(\{x,y\},p^\prime) \big\} \in \pi_y$. We let $v_{y,x}=v_{y,x}(m,\pi_y)$ denote the number of vertical pairings on~$y$ towards~$x$.

In the following, with a slight abuse of notation, we will write $n_x=n_x(m)= \frac{1}{2} \sum_{y \sim x} m_{\{x,y\}}$ for any $m \in \Mcal$ and $x \in V$.

In the following Lemma, we condition on the numbers of links at all edges and on the pairings at all vertices~$y\in V_2$, so that what remains to sample is the pairings at every~$x\in V_1$.
The Lemma shows that, with this conditionning, the numbers of double links around two different vertices~$x,\,x'\in V_1$ are conditionnally independent, and the distribution of these numbers around every point~$x\in V_1$ is obtained using the fixed points in  an Ewens distribution.

\begin{lemma}
	\label{lemma:ewens}
	For any $(a_e)_{e\in E}\in(\mathbb{N}_0)^E$, $m \in \Mcal$ and $(\pi_y)_{y \in V_2}$ with $\pi_y \in \Pcal_y(m)$ for any $y \in V_2$, 
\begin{equation} \label{eq:lemmaewens}
	\P_{G, U,  N, \beta}
	\Big(\,
	\forall e\in E,\ \tilde{k}_e= a_e
	\ \Big|\ 
	{m},\,(\pi_y)_{y\in V_2}
	\,\Big)
	\ =\ 
	\prod_{x\in V_1}
	p\Big(
	N/2,\,n_x,\,(v_{y,x})_{y\sim x},\,\big(a_{\{x,y\}}\big)_{y\sim x}
	\Big)
	\,,
\end{equation}
	where we defined
	\begin{equation}
	\label{eq:defpEwens}
	p(\theta,\,n,\,v_1,\,\ldots,\,v_r,\,a_1,\,\ldots,\,a_r)
	\ :=\ 
	P_{\theta,\,n}^{{\rm Ewens}}
	\big(
	\forall j\leq r,\ 
	|{\rm FP}(\sigma)\cap A_j|= a_j
	\big)
	\,,
	\end{equation}
	where~$r$ is the degree of the vertex~$x$ in the graph~$G$ and the sets~$A_j$ are arbitrarily chosen pairwise disjoint subsets of~$[n]$ with~$|A_j|=v_j$ for every~$j\leq r$.
\end{lemma}

\begin{proof}	
{Let $(a_e)_{e\in E}\in(\mathbb{N}_0)^E$ and $m \in \Mcal$. The proof of the lemma proceeds in two steps. In the first step, we fix a single vertex $x \in V$ and study the joint distribution of double links on the edges adjacent to $x$, conditioned on $m$ and on the pairings at all vertices $z \in V$ with $z \neq x$. More precisely, we will show that for any $(\pi_z)_{z \in V \setminus \{x\}}$ with $\pi_z \in \Pcal_z(m)$ for all $z \in V \setminus \{x\}$, {we have}
\begin{equation} \label{eq:proofEwensfirststep}
\P\Big(\tilde{k}_{\{x,y\}} =a_{\{x,y\}} \, \forall y \sim x \ \Big| \ m, (\pi_z)_{z \in V \setminus \{x\}}\Big)
= p\Big(N/2,\,n_x,\,(v_{y,x})_{y\sim x},\,\big(a_{\{x,y\}}\big)_{y\sim x}\Big)\,.
\end{equation}
In the second step we will then provide an argument based on induction to deduce \eqref{eq:lemmaewens} from  \eqref{eq:proofEwensfirststep}.}

{Let us begin with the proof of \eqref{eq:proofEwensfirststep} and let $(\pi_z)_{z \in V \setminus \{x\}}$ with $\pi_z \in \Pcal_z(m)$ for all $z \in V \setminus \{x\}$. Given $m$ and $(\pi_z)_{z \in V \setminus \{x\}}$ we obtain a collection of closed loops, which do not visit $x$ and of \lq half-loops\rq{}, namely open paths starting and ending at $x$ and not visiting $x$ in between. The number of such half-loops is given by $n_x$. 
%We will now give and justify a procedure how to construct the pairings at $x$ with the correct conditional distribution. 
From the definition of the RPM it follows that, conditioned on $m$ and $(\pi_z)_{z \in V \setminus \{x\}}$, the probability of any pairing $\pi_x \in \Pcal_x(m)$ at $x$ is proportional to~$N$ to the power of the number of loops in the resulting configuration. More precisely, if for any $\pi_x \in \Pcal_x(m)$ we denote by $c(\pi_x)$ the number of loops touching $x$ in the configuration $(m,(\pi_z)_{z \in V})$ and if we denote by~$\bar{\mathbb{P}}$ the conditional probability when~$m$ and~$(\pi_z)_{z\in V\setminus\{x\}}$ are fixed, then for any $\tilde{\pi}_x \in \Pcal_x(m)$, {we have}
\begin{equation} \label{eq:correctdistribution}
\bar{\mathbb{P}}(\pi_x=\tilde{\pi}_x) = \frac{N^{c(\tilde{\pi}_x)}}{\tilde{\Z}}
\end{equation}
for some normalising constant $\tilde{\Z}=\tilde{\Z}(n_x,N)$. 
We will now explain a procedure how to construct the pairings at $x$ with conditional distribution given by \eqref{eq:correctdistribution}. We proceed by 
\begin{itemize}
\item[(i)] labelling the half-loops incident to~$x$ with integers from~$1$ to~$n_x$,
\item[(ii)] choosing uniformly and independently an orientation for each of these half-loops, independently of everything else. Then the situation at~$x$ is that we have~$n_x$ ingoing links and~$n_x$ outgoing links, each half-loop starting with an outgoing link and finishing with an ingoing link,
\item[(iii)] drawing a permutation of the set~$\{1,\,\ldots,\,n_x\}$ according to the Ewens distribution of parameter~$\theta=N/2$, independently of everything else,
\item[(iv)] pairing the ingoing link of the half-loop number~$i$ with the outgoing link of the half-loop number~$\sigma(i)$ for every~$i\in\{1,\,\ldots,\,n_x\}$, as drawn in Figure~\ref{fig:Ewens},
\item[(v)] eventually, forgetting the orientation of the loops, so to obtain a random pairing~$\pi_x$ of the links at the site~$x$.
\end{itemize}

\begin{figure}
\centering
\begin{subfigure}[t]{.3\textwidth}
  \centering
  \includegraphics[width=\textwidth]{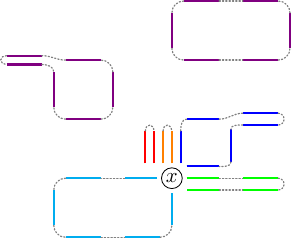}
  \caption{}
  \label{fig:sub1b}
\end{subfigure}
\hspace{6em}
\begin{subfigure}[t]{.4\textwidth}
  \centering
  \includegraphics[width=\textwidth]{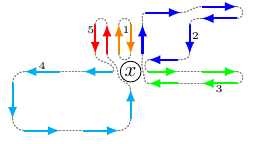}
  \caption{}
  \label{fig:sub2b}
\end{subfigure}
\caption{(a) A link configuration $m \in \Mcal$ and a pairing configuration $\pi_z \in \Pcal_z(m)$ at all vertices $z \neq x$. The pairings are represented by the dotted gray lines. We obtain a collection of two closed loops not visiting $x$ and of $5=n_x$ open paths starting and ending at $x$ not visiting $x$ in between. The colours in the figure are only for illustration. (b) Illustration of the pairing construction at $x$. The open paths were labelled from $1$ to $5$ and for each of the paths we chose an orientation uniformly at random. Our permutation $\sigma \in S_5$ is given by $\sigma(1)=1$, $\sigma(2)=3$, $\sigma(3)=2$, $\sigma(4)=5$ and $\sigma(5)=4$. The ingoing-link of path $i$ is paired with the outgoing link of path $\sigma(i)$. We obtain three closed loops at $x$ and one of them is a loop of length two, i.e., a double link.}
\label{fig:Ewens}
\end{figure}

We will now show that under this procedure, the distribution of $\pi_x$ under $\bar{\P}$ is indeed given by \eqref{eq:correctdistribution}. For this, let~$\mathcal{O}$ denote the set of the~$2^{n_x}$ possible orientations of the half-loops at~$x$. Moreover, for every pairing~$\tilde{\pi}_x \in \Pcal_x(m)$ of the links incident to~$x$, let~$\mathcal{O}(\tilde{\pi}_x)\subset\mathcal{O}$ denote the set of all orientations of the half-loops at~$x$ which are compatible with~$\tilde{\pi}_x$, in the sense that the orientations are consistent along the cycles of~$\tilde{\pi}_x$. Then we have~$|\mathcal{O}(\tilde{\pi}_x)|=2^{c(\tilde{\pi}_x)}$.  Therefore, letting~$O$ be the random orientation of the half-loops chosen during the above procedure, the conditional probability to obtain a given pairing~$\tilde{\pi}_x \in \Pcal_x(m)$ is given by}

	{
	\begin{align*}
		\bar{\mathbb{P}}(\pi_x=\tilde{\pi}_x)
	%	&\ =\ 
	%	\sum_{o\in\mathcal{O}}
	%	\bar{\mathbb{P}}(\pi_x=\tilde{\pi}_x,\,O=o)
	%	\ =\ 
%	\sum_{o\in\mathcal{O}(\tilde{\pi}_x)}
	%	\bar{\mathbb{P}}(\pi_x=\tilde{\pi}_x,\,O=o)\\
		&\ =\ 
		\sum_{o\in\mathcal{O}(\tilde{\pi}_x)}
		\bar{\mathbb{P}}(O=o)
		\,\bar{\mathbb{P}}(\pi_x=\tilde{\pi}_x\ |\ O=o)
		\ =\ 
		\sum_{o\in\mathcal{O}(\tilde{\pi}_x)}
		\frac{1}{2^{n_x}}
		\times
		\frac{(N/2)^{c(\tilde{\pi}_x)}}{Z(N/2,\,n_x)}\\
		&\ =\ 
		2^{c(\tilde{\pi}_x)}
		\times
		\frac{1}{2^{n_x}}
		\times
		\frac{(N/2)^{c(\tilde{\pi}_x)}}{Z(N/2,\,n_x)}
		\ =\ 
		\frac{N^{c(\tilde{\pi}_x)}}{2^{n_x}Z(N/2,\,n_x)}
		\,,
	\end{align*}
		}
as required. Let us now consider the event
	$$
	\mathcal{A}_x
	\ =\ \big\{w \in \widetilde\Wcal : \, \tilde{k}_{\{x,y\}}= a_{\{x,y\}}\, \forall y\sim x \big\}.
	$$
With the above construction, a half-loop of length two incident to~$x$ (which corresponds to a vertical pairing at~$y$ on the edge~$\{x,\,y\}$) with a certain label, say~$i$, becomes a closed loop of length two if and only if~$i$ is a fixed point of the permutation drawn in the above procedure.
	Thus, for every~$x\in V_1$, we have
$$
		\mathbb{P}_{G, U,  N, \beta}
		\Big(\,\mathcal{A}_x
		\ \Big|\ m,\,(\pi_y)_{y\in V\setminus\{x\}}\,\Big)
		\ =\ 
		p\Big(
		N/2,\,n_x,\,(v_{y,x})_{y\sim x},\,\big(a_{\{x,y\}}\big)_{y\sim x}
		\Big)
		\,,
$$
with~$p$ the function defined by~\eqref{eq:defpEwens}. This proves \eqref{eq:proofEwensfirststep}.

	Now, to conclude the proof, it is enough to show that, for every~$W\subseteq V_1$, we have
	\begin{equation}
		\label{eq:indepE}
		\mathbb{P}_{G, U,  N, \beta}
		\Bigg(\,
		\bigcap_{x\in W} \mathcal{A}_x
		\ \Bigg|\ {m},\,(\pi_y)_{y\in V_2}
		\,\Bigg)
		\ =\ \prod_{x\in W}
		\mathbb{P}_{G, U,  N, \beta}
		\Big(\,\mathcal{A}_x
		\ \Big|\ {m},\,(\pi_y)_{y\in V_2}
		\,\Big)
		\,.
	\end{equation}
	We proceed by induction on the size of~$W$.
	The claim is obvious if~$W$  has cardinality $1$. 
	Assume now that~$0\leq r<|V_1|$ is such that~\eqref{eq:indepE} is true for every~$W\subset V_1$ with~$|W|=r$, and let~$W\subset V_1$ with~$|W|=r+1$.
{Choosing an arbitrary point~$x\in W$ and using the tower property of conditional expectations, we may write
	\begin{equation} \label{eq:inductionstep}
	\begin{aligned}
		\mathbb{P}
		\bigg(\,
		\bigcap_{z\in W} \mathcal{A}_z
		\ \Big|\ {m},\,(\pi_y)_{y\in V_2}
		\,\bigg)
		&\ =\ 
		\mathbb{E}
		\Bigg[\,
		\mathbb{P}
		\bigg(
		\bigcap_{z\in W} \mathcal{A}_z
		\ \Big|\ 
	{m},\,(\pi_y)_{y\in V\setminus\{x\}}
		\bigg)
		\ \bigg|\ 
		{m},\,(\pi_y)_{y\in V_2}
		\,\Bigg]\\
		&\ =\ 
		\mathbb{E}
		\Bigg[
		\bigg(\prod_{z\in W\setminus\{x\}}
		\mathbbm{1}_{\mathcal{A}_z}\bigg)
		\mathbb{P}
		\Big(\,
		\mathcal{A}_x
		\ \Big|\ 
	{m},\,(\pi_y)_{y\in V\setminus\{x\}}\,
		\Big)
		\ \bigg|\ 
	{m},\,(\pi_y)_{y\in V_2}
		\,\Bigg] \\ 
	%	&\ =\ \mathbb{E}\Bigg[\bigg(\prod_{z\in W\setminus\{x\}}\mathbbm{1}_{\mathcal{A}_z}\bigg)p\Big(N/2,\,n_x,\,(v_{y,x})_{y\sim x},\,\big(a_{\{x,y\}}\big)_{y\sim x}\Big)\ \Bigg|\ m,\,(\pi_y)_{y\in V_2}\,\Biggg]\\
		&\ =\ 
		\mathbb{P}
		\Bigg(
		\bigcap_{z\in W\setminus\{x\}}
		\mathcal{A}_z
		\ \Bigg|\ 
		{m},\,(\pi_y)_{y\in V_2}
		\,\Bigg)
		\times
		\mathbb{P}
		\Big(\,
		\mathcal{A}_x
		\ \Big|\ 
	{m},\,(\pi_y)_{y\in V\setminus\{x\}}\,
		\Big),
	\end{aligned}
	\end{equation}
where, in the second equality, we used the fact that, for every~$z\in V_1\setminus\{x\}$, since~$z$ is not a neighbour of~$x$, the event~$\mathcal{A}_z$ is measurable with respect to the sigma-algebra generated by the vertical pairings~$(\pi_y)_{y\in V\setminus\{x\}}$ and in the last equality we applied \eqref{eq:proofEwensfirststep}. Using the induction hypothesis, we deduce that~\eqref{eq:indepE} holds for~$W$, which finishes the proof by induction and,  consequently,   the proof of the lemma. }
\end{proof}

With the result of Lemma \ref{lemma:ewens} at hand, we are now  ready prove the following stochastic domination.

\begin{lemma}
	\label{lemma:upperstochdomination}
For any $x\in V_1$, {let us denote by} $X_x:{\widetilde{W}} \to \N_0$ {the number of double links around the {vertex~$x$}, that is to say,}
	$$
	\forall w \in {\widetilde{W}}, \qquad X_x(w)
	\ :=\ 
	\sum_{y\sim x} \tilde{k}_{\{x,y\}}(w).
	$$
Let further~$(Y_x)_{x\in V_1}$ be i.i.d.\ random variables taking values in~$\N_0$, with distribution given by
	\begin{equation}
	\label{eq:distribYx}
	\forall k\in\mathbb{N}\,,
	\qquad
	P(Y_x\geq k)
	\ =\ 
	\min\bigg\{1,\ 
	\sum_{j=k}^\infty
		\frac{\max\big\{1,N/2\big\}^j}{j!} \bigg\}
	\,.
	\end{equation}
Then, {we have the stochastic domination}~$(X_x)_{x\in V_1}\preceq(Y_x)_{x\in V_1}$.
\end{lemma}

\begin{proof}
To begin, we note that Lemma~\ref{lemma:ewens} implies that, conditionally on $m \in \Mcal$ and $(\pi_y)_{y\in V_2}$ with $\pi_y \in \Pcal_y(m)$ for all $y \in V_2$, the variables~$(X_x)_{x\in V_1}$ are independent and, for every~$x\in V_1$ and every~$k\in\mathbb{N}$, we have that,
	\begin{align*}
		\mathbb{P}_{G, U,  N, \beta}
		\Big(\,
		X_x\geq k
		\ \Big|\ 
		{m},\,(\pi_y)_{y\in V_2}
		\,\Big)
		&\ =\ 
		\sum_{(a_y)_{y\sim x}\,:\,\sum_{y\sim x} a_y\,\geq\,k}
		\,
		p\left(
		\frac{N}{2},\,n_x,\,(v_{y,x})_{y\sim x},\,(a_y)_{y\sim x}
		\right)
		\\
		&\ =\ 
		q\left(\frac{N}{2},\,n_x,\,\sum_{y\sim x}v_{y,x},\,k
		\right)
		\,,
	\end{align*}
	where, for every~$\theta\in\mathbb{R}_0^+$ and~$n,\,v,\,k\in\mathbb{N}$, we defined
	$$
	q(\theta,\,n,\,v,\,k)
	\ :=\ 
	P_{\theta,\,n}^{\text{Ewens}}\Big(\big|{\rm FP}(\sigma)\cap\{1,\,\ldots,\,v\}\big|\,\geq\,k\Big)
	\,.
	$$
	Now note that, for every~$\theta\in\mathbb{R}_0^+$,~$n\geq k\geq 0$ and~$v\geq 0$, we have
	\begin{align*}
		q(\theta,\,n,\,v,\,k)
		&\ \leq\ 
	P_{\theta,\,n}^{\text{Ewens}}\Big(\big|{\rm FP}(\sigma)\big|\,\geq\,k\Big)
		\ =\ 
		\frac{1}{Z(\theta,\,n)}
		\sum_{j=k}^n
		\binom{n}{j}
		\theta^{j}
		\sum_{\sigma\in S_{n-j}}
		\theta^{c(\sigma)}
		\mathbbm{1}_{\{{\rm FP}(\sigma)=\varnothing\}}
		\\
		&\ \leq\ 
		\frac{1}{Z(\theta,\,n)}
		\sum_{j=k}^n
		\binom{n}{j}
		\theta^{j}
		Z(\theta,\,n-j)
		\ =\ 
		\sum_{j=k}^n
		\frac{\theta^j (1+n-1)\ldots(1+n-j)}{j!(\theta+n-1)\ldots(\theta+n-j)}
		\\
		&\ \leq\ 
		\sum_{j=k}^n
		\frac{\theta^j}{j! \, \min\{1,\theta\}^j}
		\ =\ 
		\sum_{j=k}^n
		\frac{\max\{1,\theta\}^j}{j!}
		\ \leq\ 
		\sum_{j=k}^\infty
		\frac{\max\{1,\theta\}^j}{j!}
		\,,
	\end{align*}
	where in the fourth step we applied~\eqref{eq:zewens}.
Since~$q(\theta,\,n,\,v,\,k)$ is a probability (and is~$0$ if~$k>n$), we deduce that for every~$n,\,v,\,k\geq 0$, we have
	$$q\left(\frac{N}{2},\,n,\,v,\,k\right)
	\ \leq\ 
	\min\bigg\{1,\ 
	\sum_{j=k}^\infty
		\frac{\max\big\{1,N/2\big\}^j}{j!} \bigg\}
	\ =\ 
	P(Y_x\geq k)\,.
	$$
	Thus, we proved the stochastic domination conditionally on~${m}$ and~$(\pi_y)_{y\in V_2}$, with a stochastic upper bound which only depends on~$N$, hence the claimed uniform stochastic upper bound follows.
\end{proof}

\subsubsection{Concluding proof of the upper bound in Theorem \ref{theo:coexistence}}\label{suse:conclusion-ub}

We are now in a position to prove the upper bound on the density of any microscopic loop.

\begin{proof}[Proof of Proposition~\ref{prop:coexistence}]
Let~$k\in\mathbb{N}$ and~$\gamma\in\Sigma(\mathcal{L})$ with~$\alpha(\gamma)=2k$.
{Upon exchanging~$V_1$ and~$V_2$, we may assume that~$|V_1\cap\mathrm{supp} \,(\gamma)|\leq k$.}
Let~$a\in\mathbb{N}$. 
Applying Lemma~\ref{lemma:colourswitch} we have that,
\begin{equation} \label{eq:finalizingtheproof}
\begin{aligned}
\mu_a(\gamma)
% & \leq \psi(\gamma)^a \, \, \Ecal\bigg(\prod_{e \in E} \frac{(k_e)!}{(k_e-a \, q_e(\gamma))!} \, \mathbbm{1}_{\{k_e \geq a \times q_e(\gamma)\}}\bigg) \\
& \leq \psi(\gamma)^a \, \, \Ecal\bigg(\prod_{x \in V_1} \prod_{y \sim x} \frac{k_{\{x,y\}}!}{(k_{\{x,y\}}-a \, q_{\{x,y\}}(\gamma))!} \, \mathbbm{1}_{\{k_{\{x,y\}} \geq a \times q_{\{x,y\}}(\gamma)\}}\bigg) \\
& \leq \psi(\gamma)^a \, \mathbb{E}\bigg(\prod_{x\in V_1 \cap\, \text{supp}\,(\gamma)}\frac{X_x!}{(X_x-a\times n_x(\gamma))!} \, \mathbbm{1}_{\{X_x \geq a\times n_x(\gamma)\}} \bigg) \\
& \leq \psi(\gamma)^a \, \prod_{x\in V_1 \cap\, \text{supp}\,(\gamma)} e^{\max\{1,\frac{N}{2}\}} \, E\bigg(\frac{Z!}{(Z-a \times n_x(\gamma))!} \mathbbm{1}_{\{Z \geq a \times n_x(\gamma)\}}\bigg) \\
& = \psi(\gamma)^a \, \prod_{x\in V_1 \cap\, \text{supp}\,(\gamma)} e^{\max\{1,\frac{N}{2}\}} \, \max\{1,\frac{N}{2}\}^{a \times n_x(\gamma)}  \\
& \leq \psi(\gamma)^a \, \max\{e,e^{\frac{N}{2}}\}^{k} \, \max\{1,\frac{N}{2}\}^{a \times k} \\
& \leq \big( \psi(\gamma) \, \max\{e,e^{\frac{N}{2}}\}^k \, \max\{1,\frac{N}{2}\}^{k} \Big)^a = \lambda(\gamma)^a,
\end{aligned}
\end{equation}
where in the second step we applied Lemma \ref{lemma:equivalence} and then we used that $$
\frac{n_1! n_2!}{(n_1-k_1)! (n_2-k_2)!} \leq \frac{(n_1+n_2)!}{(n_1+n_2-k_1-k_2)!}
$$ 
for any $n_1,n_2,k_1,k_2 \in{\N_0}$.
In the third step we applied Lemma \ref{lemma:upperstochdomination}. Here, $Z$ is a random variable that is Poisson distributed with parameter $\max\{1,\frac{N}{2}\}$ and by $E$ we denote the expectation with respect to $Z$.  
This concludes the proof of the proposition.
\end{proof}

\subsection{Proof of the lower bound in Theorem \ref{theo:coexistence}}\label{sect:lowerbound}
In this section we prove the lower bound of Theorem \ref{theo:coexistence}. Throughout the section, we fix parameters $\beta, N >0$ and $M,\Delta\in\N$, and a function  $U:\N_0 \to \R^+$ that is $M$-decaying and fully supported, i.e., $U(n)>0$ for all $n\in\N$.
\begin{proposition}\label{prop:lb-beta}
 Fix $k\in\N$. There exists $\varepsilon=\varepsilon(k,\Delta,\beta,N,U)$ such that for every finite connected graph $G=(V,E)$ with maximal degree $\Delta$ and every $\gamma\in\Sigma(\mathcal{L})$ with $\alpha(\gamma)=2k$,
	\begin{equation} 
		\mathcal{E}_{G,U,N,\beta}[k_\gamma]>\varepsilon\, .
	\end{equation}
\end{proposition}
{The above proposition then implies the lower bound of Theorem~\ref{theo:coexistence}, by simply considering the loops of length~$2k$ which are the concatenation of~$k$ times the same loop of length~$2$.} 

We  first prove  the above proposition in the framework of the RPM. Proposition \ref{prop:lb-beta} will then follow immediately from Lemma \ref{lemma:equivalence}. 
We begin by showing the following lemma, which states that the local time has finite moments.
\begin{lemma}\label{lemma:local-time-beta}
For all $m \in \N$, there exists $b=b(m,\Delta,\beta,N, U)<\infty$ such that, for every {finite} connected graph $G=(V,E)$ and every $o\in V$ having degree at most $\Delta$,
	\begin{equation} \label{eq:localtimefiniteexp}
		\mathbb{E}_{G,U,N,\beta}[n_o^m]\le b\, .
	\end{equation}
\end{lemma}
%\begin{remark}
%	Notice that the bound in \eqref{eq:localtimefiniteexp} depends on $U$ only through $M$. Lemma \ref{lemma:local-time-beta} thus provides a non-trivial bound on the expected local time even if the range of $U$ is finite, as soon as it is larger than $b(1,\Delta,\beta,N,M)$.
%\end{remark}
An immediate consequence of the above lemma is the following corollary. 
\begin{corollary}\label{lemma:event-A-beta}
 Fix $a\in\N$. 
 %Assume that  $R\add{=\sup\{n\in\N\: : \:U(n)>0\}}> h+ 2a\times b(1,\Delta,\beta,N,M),$	where $b$ is defined as in Lemma \ref{lemma:local-time-beta}.
  There exists $\bar n=\bar n(a,\Delta,\beta,N, U)< \infty$ such that for every {finite} connected graph $G=(V,E)$ of maximal degree $\Delta$ and any $A\subset V$ such that $|A|=a$, 
	\begin{eqnarray}
		\P_{G,U,N,\beta}(\forall x\in A,\: n_x\leq \bar n )\ge \frac12\,.
	\end{eqnarray}
\end{corollary}
\begin{proof}
	Using the union bound, Markov's inequality and \eqref{eq:localtimefiniteexp},  and choosing $\bar n =2ab$, where $b$ is as in Lemma \ref{lemma:local-time-beta} with $m=1$, we have that, 
	$$
	\P_{G,U,N,\beta}(\forall x\in A,\: n_x\leq \bar n ) \geq 1-{\sum_{x\in A}\P _{G,U,N,\beta}(n_x>2ab)}\geq \frac12.
	$$
	This concludes the proof of the corollary.
\end{proof}
We now present the proof of Lemma~\ref{lemma:local-time-beta}, which proceeds in the same spirit as the proof of \cite[Lemma~3.1]{L-T-exponential}.

\begin{proof}[Proof of Lemma \ref{lemma:local-time-beta}] 
Let $m \in \N$. We show that for any $k \in \N$, there exists $c_m(k)=c_m(k,\Delta,\beta,N,U)$ such that,  {writing} $k_m:= \lceil k^\frac{1}{m} \rceil$,
	\begin{eqnarray} \label{eq:showck}
		\P_{G,U,N,\beta}(n_o\ge k_m)\le c_m(k)
	\end{eqnarray}
	and such that
	\begin{eqnarray} \label{eq:sumckfinite}
		\sum_{k\ge 1}c_m(k)<\infty \ ,
	\end{eqnarray}
from which the desired result follows immediately.
	Fix any $k\geq 1$. We set $E_o:=\{e \in E: \, o \in e\} $ and we set $V_o := \{x \in V: x \sim o \} \cup \{o\} $. We let {$\Sigma=\{0,1\}^{E_o}$}. {Recalling that~$\mu$ denotes the non-normalized measure defined by~\eqref{eq:RPMmeasure}, we begin by writing}
	\begin{equation} \label{eq:localtimegreaterk}
		\begin{aligned}
			\mu(n_o \geq k_m) 
			& = \sum_{n=k_m}^\infty \sum_{\xi \in \Sigma} \sum_{\substack{m \in \Mcal: \\ m_e \in 2\N +\xi_e {\forall e \in E_o,} \\ n_o(m)=n}} \sum_{\pi \in \Pcal(m)} \mu((m,\pi)) \\
			& = \sum_{n=k_m}^\infty \sum_{\xi \in \Sigma} \sum_{\substack{m^\prime \in \Mcal: \\ m^\prime_e =\xi_e \, \forall e \in E_o}}  \sum_{\pi^\prime \in \Pcal(m^\prime)} \sum_{\substack{m \in \Mcal: \\ m_e=m_e^\prime \, \forall e \in E \setminus E_o,\\ m_e \in 2\N +\xi_e \, \forall e \in E_o, \\ n_o(m)=n}} \sum_{\substack{\pi \in \Pcal(m): \\ \pi_x=\pi^\prime_x \, \forall x \in V \setminus V_o}} \mu((m,\pi)) \, \prod_{x \in V_o} \frac{1}{(2n_x(m^\prime)-1)!!}\,,
		\end{aligned}
	\end{equation}
	{where we let~$(2k-1)!!:=(2k)!/(k!\,2^k)$, which corresponds to the number of pairings of~$2k$ elements.}
Let $n \geq k_m$. Let $\xi \in \Sigma$, $m^\prime \in \Mcal$ such that ${m^\prime_e} =\xi_e$ for all $e \in E_o$, and $\pi^\prime \in \Pcal(m^\prime)$. Let $m \in \Mcal$ such that $m_e =m_e^\prime$ for all $e \in E \setminus E_o$ and $m_e \in 2\N + \xi_e$ for all $e \in E_o$ and $n_o(m)=n$ and let $\pi \in \Pcal(m)$ such that $\pi_x=\pi_x^\prime$ for all $x \in V \setminus V_o$.
	 We  now compare 
	$\mu((m,\pi))$ with $\mu((m^\prime, \pi^\prime))$.
	We note that,
	\begin{itemize}
		\item[(i)] $
		\# \text{ loops in } (m,\pi) \leq \# \text{ loops in } (m^\prime,\pi^\prime)+n,
		$
		\item[(ii)] $\sum_{e \in E_o} (m_e -m_e^\prime) \leq 2n$,
		\item[(iii)] $\prod_{e \in E_o} \frac{1}{m_e!} \leq \Big(\frac{1}{\lfloor \frac{2n}{\text{deg}(o)} \rfloor!}\Big)^{\text{deg}(o)} \leq (2n)^{{\rm deg}(o)}\ \frac{(\text{deg}(o))^{2n}}{(2n)!}$.
	\end{itemize}
	From (i) - (iii) we deduce that
	\begin{equation} \label{eq:measurecomparison}
		\begin{aligned}
			\mu((m,\pi)) 
			& = \mu((m^\prime,\pi^\prime)) \, N^{\# \text{ loops in } (m,\pi) - \# \text{ loops in } (m^\prime, \pi^\prime)} \, \prod_{e \in E_o} \frac{\beta^{m_e-m_e^\prime}}{m_e!} \, \prod_{x \in V_o} \frac{U(n_x(m))}{U(n_x(m^\prime))} \\
			& \leq \mu((m^\prime,\pi^\prime))  \, \frac{c_1^n(2n)^{\text{deg}(o)}}{(2n)!} \,  \prod_{x \in V_o} \frac{U(n_x(m))}{U(n_x(m^\prime))} ,
		\end{aligned}
	\end{equation}
	where $c_1:= \max\{1,N \}\, \max\{1,\beta\}^2 \, \text{deg}(o)^2$.
	Plugging \eqref{eq:measurecomparison} into \eqref{eq:localtimegreaterk}, and using the bounds
	\begin{align}
		\big|\left\{\pi\in\mathcal{P}(m)\mid\pi_x=\pi_x',\ \forall x\in V\setminus V_o\right\}\big|\le& \prod_{x\in V_o}(2n_x(m)-1)!!\\
		\Big|\Big\{\substack{m \in \Mcal: \, m_e=m_e^\prime \, \forall e \in E \setminus E_o, \\ \, m_e \in 2\N+\xi_e \, \forall e \in E_o, \, n_o(m)=n}\Big\}\Big|\le& (2n)^{{\rm deg}(o)}
	\end{align}
	we deduce that
	\begin{equation} \label{eq:laststepbound}
		\begin{aligned}
			\mu(n_o \geq k_m) 
			& \leq \sum_{n=k_m}^\infty \frac{c_1^n(2n)^{\text{deg}(o)}}{(2n)!} \, \sum_{\xi \in \Sigma} \sum_{\substack{m^\prime \in \Mcal: \\ m^\prime_e =\xi_e \, \forall e \in E_o}}  \sum_{\pi^\prime \in \Pcal(m^\prime)} \mu((m^\prime,\pi^\prime)) \\
			& \qquad \qquad \times \sum_{\substack{m \in \Mcal: \\ m_e=m_e^\prime \, \forall e \in E \setminus E_o,\\ m_e \in 2\N +\xi_e \, \forall e \in E_o, \\ n_o(m)=n}} \sum_{\substack{\pi \in \Pcal(m): \\ \pi_x=\pi^\prime_x \, \forall x \in V \setminus V_o}} \prod_{x \in V_o } \frac{U(n_x(m))}{U(n_x(m^\prime))} \,  \frac{1}{(2n_x(m^\prime)-1)!!} \\
			& \leq \sum_{n=k_m}^\infty \frac{c_1^n(2n)^{\text{deg}(o)}}{(2n)!} \, \bigg[ \max\limits_{0 \leq i \leq j \leq {i}+n} \frac{U(j)(2j-1)!!}{U(i)(2i-1)!!}\bigg]^{\text{deg}(o)+1}\\
			& \qquad \qquad \times \sum_{\xi \in \Sigma} \sum_{\substack{m^\prime \in \Mcal: \\ m^\prime_e =\xi_e \, \forall e \in E_o}}  \sum_{\pi^\prime \in \Pcal(m^\prime)} \mu((m^\prime,\pi^\prime)) \, \, \Big|\Big\{\substack{m \in \Mcal: \, m_e=m_e^\prime \, \forall e \in E \setminus E_o, \\ \, m_e \in 2\N+\xi_e \, \forall e \in E_o, \, n_o(m)=n}\Big\}\Big|  \\
			& \leq \sum_{n=k_m}^\infty \frac{\tilde{c}_1^n}{(2n)!} \, (2n)^{2\,\text{deg}(o)} \, \sum_{\xi \in \Sigma} \sum_{\substack{m^\prime \in \Mcal: \\ m^\prime_e =\xi_e \, \forall e \in E_o}}  \sum_{\pi^\prime \in \Pcal(m^\prime)} \mu((m^\prime,\pi^\prime)) \\
			& \leq \sum_{n \geq k_m} \frac{\tilde{c}_1^n}{(2n)!} (2n)^{2\,\text{deg}(o)} \, \Z_{G,U,N,\beta}\,,
		\end{aligned}
	\end{equation}
	where $\tilde{c}_1:=  (2M)^{\Delta+1} \, c_1$.  
	In the third step of \eqref{eq:laststepbound} we used that $U$ being fast decaying implies that for any $n \in \N$ and any $0 \leq i \leq j \leq i+n$, 
	\begin{equation}
		\frac{U(j)(2j-1)!!}{U(i)(2i-1)!!} 
		\leq 2^{j-i} \, \frac{U(j)}{U(i)} \, j \, (j-1) \cdots (i+1) 
		\leq (2M)^{j-i} 
		\leq (2M)^n\,.
	\end{equation}
	Setting 
	$$
	c_{m}(k):= \sum_{n \geq k_m} \frac{\tilde{c}_1^n}{(2n)!} (2n)^{2\Delta}
	$$ 
	we deduce \eqref{eq:showck}. Since there exists $c^\prime < \infty$ such that
	$$
	c_{m}(k) \leq \frac{\tilde{c}_1^{k_m}}{(2k_m)!} \, (2k_m)^{2\Delta} \sum_{n \geq 0} \frac{\tilde{c}_1^n}{(2n)!} (n+1)^{2\Delta} \leq c^\prime \, \frac{\tilde{c}_1^{k_m}}{(2k_m)!} \, (2k_m)^{2\Delta}\,,
	$$
	and
		$$
		\sum_{k=1}^\infty \frac{\tilde{c}_1^{k_m}}{(2k_m)!} \, (2k_m)^{2\Delta} \leq \sum_{n=1}^\infty \frac{\tilde{c}_1^n}{(2n)!} \, (2n)^{2\Delta} \, n^m < \infty\,,
		$$
	we deduce \eqref{eq:sumckfinite} and the proof of the lemma is concluded. 
\end{proof}

We now present the proof of Proposition \ref{prop:lb-beta}.

	\begin{proof}[Proof of Proposition \ref{prop:lb-beta}] 
Fix $\gamma\in\Sigma(\mathcal{L})$ with $\alpha(\gamma)=2k$.  Recalling from \eqref{eq:def-supp} that we denote by $\mathrm{supp}(\gamma)$ the set of distinct vertices visited by $\gamma$, we have $a_\gamma:=|{\mathrm{supp}(\gamma)}|\le 2k$. Let $m_e(\gamma)$ denote the number of times {that} the loop $\gamma$ uses the edge $e$.
Fix $ n_0=n_0(k,\Delta, \beta,N,U) $, such that 
\begin{equation} \label{eq:boundonethird}
\P(\forall x\in {\mathrm{supp}(\gamma)},\ n_x \leq n_0) \geq \frac12\,.  
\end{equation}
The existence of $n_0$ follows from Corollary \ref{lemma:event-A-beta}, choosing $ n_0(k,\Delta, \beta,N,U) =\bar n (2k,\Delta,\beta,N,U)$.
Consider the events
	\begin{equation*}
		\mathcal{A}=\{w \in \widetilde\Wcal \, : \forall x\in {\mathrm{supp}(\gamma)},\ n_x \leq n_0, \, \tilde k_\gamma=0 \}\,, \quad \  \mathcal{B}=\{w \in \widetilde\Wcal \, : \, \forall x\in {\mathrm{supp}(\gamma)}, \ n_x\leq n_0+n_x(\gamma), \, \tilde k_\gamma>0 \}\,.
	\end{equation*}
We introduce the map $\phi:\mathcal{A}\to \mathcal{B}$, which for any $w \in \Acal$ acts by inserting $m_e(\gamma)$ links on the edge $e$ such that they have the highest labels, and by pairing them with an arbitrary deterministic rule such that we obtain one additional cycle, which under the map $\vartheta$ (defined in \eqref{eq:def-vartheta}) reduces to $\gamma$.
%{to form the loop~$\gamma$, choosing the order of the links used on each edge with an arbitrary deterministic rule}. 
For any $w \in \Acal$ we then have that
\begin{equation} \label{eq:imagephi}
	\mu(\phi(w)) =\mu(w) \, \beta^{2k} N\prod_{\substack{e\in E\\ m_e(\gamma)>0}}\frac{1}{(m_e(w)+1)\cdots(m_e(w)+m_e(\gamma))} \, \prod_{x\in {\mathrm{supp}(\gamma)}}\frac{U(n_x(w)+n_x(\gamma))}{U(n_x(w))}\  
	\geq c \, \mu(w)\,,
\end{equation}
where
\begin{equation*}
c=c(k,\Delta,\beta,N,U) := \frac{\beta^{2k} N}{[(2n_0+1)\cdots(2n_0+2k)]^{2k}} \, \bigg(\min_{a\le n_0,\,j\le k}\frac{U(a+j)}{U(a)} \bigg)^{2k}.
\end{equation*}
Since $\phi$ is an injection, we deduce from \eqref{eq:imagephi} that 
\begin{equation}
	\begin{split}
 \P(\Bcal) \geq \P(\phi(\Acal)) \geq c \, \P(\Acal) &= c \, \big(\P(\forall x\in {\mathrm{supp}(\gamma)},\ n_x \leq n_0)-\P(\forall x\in {\mathrm{supp}(\gamma)},\ n_x \leq n_0,\ {\tilde k}_\gamma>0)\big) \\
 &\ge c{\big(} \P(\forall x\in {\mathrm{supp}(\gamma)},\ n_x \leq n_0)-\P(\Bcal){\big)}
\end{split}
\end{equation}
and the latter, together with \eqref{eq:boundonethird}, implies
\begin{equation} \label{eq:implying}
\P(\Bcal) \geq \frac{c}{2(1+c)} \, .
\end{equation}
Applying Lemma \ref{lemma:equivalence}, we deduce from \eqref{eq:implying} that
$$
\mathcal{E}(k_\gamma)=\E(\tilde{k}_\gamma)\ge\P(\tilde{k}_\gamma>0)  \geq \P(\Bcal) \geq \frac{c}{2 (1+c)}\,,
$$
which concludes the proof of the proposition.
\end{proof}

%\begin{remark}\label{rem:constants}
%
%The constants~$c_1$ and~$c_2$ {introduced} in Theorem~%\ref{theo:coexistence} are not optimal: when~$k\to\infty$ the former explode%s much faster than the number of possible loops of length less than~$k$, while the latter tends to~$0$ very fast.
%We expect that, when~$\beta\to\infty$, that is to say, in the regime where the %{local time} at each vertex tends to infinity,  the number of \text%color{blue}{occurrences} of any fixed loop  should converge to a Poisson random variable with a fixed parameter, in analogy with the situation of small components in random graphs.
%Our technique could help to establish such a Poisson behaviour, but this would require %new ingredients, such as a control of how the links at one vertex split on the edges incident to this vertex, and of which fraction of these links form \textit{vertical pairings}, as defined in Section~\ref{suse:ewens}.
%
%\end{remark}

\section{Proof of Theorem \ref{theo:nomacroscopicloops}}
In this section we present the proof of Theorem \ref{theo:nomacroscopicloops}, which {follows} from Proposition \ref{prop:correlationdecay} and Proposition \ref{prop:fromspinstoloops} below. 

\subsection{Extension of the McBryan and Spencer proof}
Throughout this section we fix $d=2$ and $L>0$.
The next proposition states that the spin-spin correlation $\langle S^1_x S_y^1 S_x^2 S_y^2 \rangle_{G_L,N,\beta}$ admits polynomial decay for any $N \in \N$ with $N>1$ and any $\beta \in \R^+$. The proof of the proposition is an adaptation of the proof of polynomial decay for $\langle S^1_x S_y^1\rangle_{G_L,N,\beta}$ given in  \cite{FriedliVelenik,McBryanSpencer}.
	\begin{proposition}
		\label{prop:correlationdecay}
For any $\beta \in \R^+$ and $N \in \N$ with $N>1$, there exists $c =c(\beta,N) \in (0,1)$ such that for any $x, y \in \Z^2$,  
		\begin{equation} \label{eq:correlationdecay}
			\lim_{L \to \infty} \, \big|\langle S^1_x S_y^1 S_x^2 S_y^2 \rangle_{G_L,N,\beta}\big|  \leq \frac{1}{8} \, |x-y|^{-c}.
		\end{equation}
	\end{proposition}
	\begin{proof} 
		To begin, we let $N=2$ and fix $\beta \in \R^+$. 
		We parametrize the unit sphere by angles such that $S_x=(\cos \theta_x, \sin \theta_x)$, where $\theta_x \in [0,2\pi)$ for any $x \in \Lambda_L$. 
		Using trigonometric identities, the invariance of the measure under simultaneous rotation of all spins and the fact that $|\text{Re}(z)| \leq |z|$ for all $z \in \C$, we obtain that for any $L>0$ and any $x,y \in \Lambda_L$,
		\begin{equation} \label{eq:spins1}
			\begin{aligned}
				\big| \langle S_x^1 S_y^1 S_x^2 S_y^2 \rangle_{G_L,N,\beta} \big|
				& = \frac{1}{8} \, \big|\langle \cos(2 (\theta_x-\theta_y)) - \cos(2 (\theta_x+\theta_y)) \rangle_{G_L,N,\beta} \big| \\
				& = \frac{1}{8} \, \big| \langle \cos(2 (\theta_x-\theta_y)) \rangle_{G_L,N,\beta}\big| \\
				& \leq \frac{1}{8} \, \big| \langle e^{2i(\theta_x-\theta_y)} \rangle_{G_L,N,\beta} \big|.
			\end{aligned}
		\end{equation}
		The derivation of the upper bound \eqref{eq:correlationdecay} now follows analogously to the proof of \cite[Theorem 9.12]{FriedliVelenik}, in which an upper bound on the quantity $| \langle e^{i(\theta_x-\theta_y)} \rangle_{G_L,N,\beta}  |$ instead of $| \langle e^{2i(\theta_x-\theta_y)} \rangle_{G_L,N,\beta}  |$ was derived. More precisely, following \cite{FriedliVelenik} and shifting the integration of the angles to the complex plane, there exists $c=c(\beta) \in (0,\infty)$ such that for any $L>0$ and any $x,y \in \Lambda_L$,   
		\begin{equation} \label{eq:upperboundGreen}
			\big| \langle e^{2i(\theta_x-\theta_y)} \rangle_{G_L,N,\beta}  \big| 
			\leq e^{-c \, \big( g_L(x,x)-g_L(x,y)-g_L(y,x)+g_L(y,y)\big)},
		\end{equation}
		where $g_L(\cdot,\cdot)$ denotes the Green function in $G_L$ of the simple random walk $X=(X_k)_{k \geq 0}$ on $\Z^2$  defined by 
		$$
		\forall x,y \in \Lambda_L, \quad \quad g_L(x,y):= \mathbf{E}_x\big( \sum_{n=0}^{\tau_L-1} \mathbbm{1}_{\{X_n=y\}}\big),
		$$
		where $\tau_L:=\inf\{k \geq 0 \, : \, X_k \in \Z^2 \setminus \Lambda_L\}$ and $\mathbf{E}_x$ denotes the expectation of the simple random walk on $\Z^2$ starting at $x$.
		By \cite[Theorem 1.6.2]{LawlerPotentialKernel} {we have}
		\begin{equation} \label{eq:potentialkernel}
			\lim_{L \to \infty} g_L(x,x)-g_L(x,y) = \lim_{L \to \infty} g_L(y,y)-g_L(y,x) = \frac{2}{\pi} \, \log||x-y||_2 +O(1) \quad \text{as } ||x-y||_2 \to \infty.
		\end{equation}
		From \eqref{eq:spins1}, \eqref{eq:upperboundGreen} and \eqref{eq:potentialkernel} we conclude \eqref{eq:correlationdecay} in the case $N=2$. 
		
		The case $N>2$ is handled by a straightforward generalization as explained in \cite{McBryanSpencer}. Namely, we parametrize the $N$-sphere by angles $\theta^{(1)}, \dots,\theta^{(N-2)}, \psi$ with $|\theta^{(r)}| \leq \frac{\pi}{2}$ for all $r \in [N-2]$ and $|\psi| \leq \pi$ in such a way that only the first two components \smash{$S^{(1)}$}, \smash{$S^{(2)}$} of a unit spin vector depend on $\psi$. We can then apply the method of \cite{FriedliVelenik} as in the previous case performing a shift of integration only with respect to $\psi$.
	\end{proof}

\subsection{From spin correlations to connection probabilities}
In this section, we fix $\beta,L \in \R^+$ and $d,N \in \N$ with $N \geq 2$. We let $G=G_L$ and $U=U_N^s$, where $U_N^s$ was defined in \eqref{eq:SpinWeightfunction}. For lighter notation, we will omit all sub-scripts. 
The next proposition gives an upper and a lower bound on the spin correlation $\langle S_x^1 S_y^1 S_x^2 S_y^2 \rangle$ for any $x \neq y \in \Lambda_L$ in terms of the probability of observing a loop connecting $x$ and $y$ in the corresponding random walk loop soup. {Recall from \eqref{eq:measure} that we denote by $\Ecal$ the expectation under the measure $\Pcal$.}
\begin{proposition}
	\label{prop:fromspinstoloops}
For any $m \in \N$ and any $x \neq y \in \Lambda_L$, 
	\begin{equation} \label{eq:mainequation}
		c_1 \, \Pcal(x \leftrightarrow y)^{1+\frac{1}{2^{(m-1)}}} \leq \langle S_x^1 S_y^1 S_x^2 S_y^2 \rangle \leq  \frac{1}{2N} \, \Pcal(x \leftrightarrow y),
	\end{equation}
	where
$c_1=c_1(\beta,N,m)>0$.
\end{proposition}
\begin{proof}
	Fix $m \in \N$ and let $x \neq y \in \Lambda_L$.
For every $\omega \in \Omega$ and $z \in \Lambda_L$, let us write $\tilde{n}_z(\omega):= n_z(\omega) + \frac{N}{2}$.
	Recall from Section \ref{sect:upperbound} that for any $\ell \in \Lcal$ we denote by $n_x(\ell)$ the number of visits of the loop $\ell$ at $x$.
	From \cite[Theorem A.1, Lemma 5.4, Theorem 3.14]{QuitmannTaggi} it is known that the correlation of spins can be written as
	\begin{equation} \label{eq:identity}
		\begin{aligned}
			\langle S_x^1 S_y^1 S_x^2 S_y^2 \rangle 
			& = \frac{2}{N} \, \Ecal\bigg( \sum_{j=1}^{|\omega|} n_x(\ell_j) \, n_y(\ell_j) \, \frac{U_N^s(n_x(\omega)+1)}{U_N^s(n_x(\omega))} \, \frac{U_N^s(n_y(\omega)+1)}{U_N^s(n_y(\omega))} \bigg) \\
			& = \frac{1}{2N} \, \Ecal\bigg( \sum_{j=1}^{|\omega|} n_x(\ell_j) \, n_y(\ell_j) \, \frac{1}{\tilde{n}_x(\omega) \, \tilde{n}_y(\omega)}  \bigg),
		\end{aligned}
	\end{equation}
	where in the last step we used the definition of $U_N^s$. Using that $\sum_{j=1}^{|\omega|} n_x(\ell_j) \, n_y(\ell_j) \leq \mathbbm{1}_{\{ x \leftrightarrow y\}}(\omega) n_x(\omega) n_y(\omega)$ for any $\omega \in \Omega$, we deduce from \eqref{eq:identity} the upper bound in \eqref{eq:mainequation}.
	We now derive the lower bound. Applying the Cauchy-Schwarz inequality iteratively $m$ times, we have that
	\begin{equation} \label{eq:upperbound}
		\begin{aligned}
			\Pcal(x \leftrightarrow y) 
			& = \Ecal\Big( \mathbbm{1}_{\{x \leftrightarrow y\}} \frac{1}{\sqrt{\tilde{n}_x \tilde{n}_y}} \, \sqrt{\tilde{n}_x \tilde{n}_y}\Big) \\
			& \leq \Ecal\Big( \mathbbm{1}_{\{x \leftrightarrow y\}} \,  \frac{1}{\tilde{n}_x \tilde{n}_y} \Big)^{\frac{1}{2}} \, \Ecal\Big(\mathbbm{1}_{\{x \leftrightarrow y\}} \, \tilde{n}_x \tilde{n}_y \Big)^{\frac{1}{2}} \\
			& \leq \Ecal\Big( \mathbbm{1}_{\{x \leftrightarrow y\}} \,  \frac{1}{\tilde{n}_x \tilde{n}_y} \Big)^{\frac{1}{2}} \, \, \Pcal(x \leftrightarrow y)^{\frac{1}{2}-\frac{1}{2^{m}}} \, \, \Ecal\Big((\tilde{n}_x \tilde{n}_y)^{2^{m-1}}\Big)^\frac{1}{2^m} \\
			& \leq \Ecal\Big( \mathbbm{1}_{\{x \leftrightarrow y\}} \,  \frac{1}{\tilde{n}_x \tilde{n}_y} \Big)^{\frac{1}{2}} \, \, \Pcal(x \leftrightarrow y)^{\frac{1}{2}-\frac{1}{2^{m}}} \, \,
\Ecal\Big((\tilde{n}_x)^{2^{m}}\Big)^{\frac{1}{2^{m+1}}}
\Ecal\Big((\tilde{n}_y)^{2^{m}}\Big)^{\frac{1}{2^{m+1}}}\,,
		\end{aligned}
	\end{equation}
where in the last step we applied the Cauchy-Schwarz inequality.
	From \eqref{eq:identity} we further observe that 
	\begin{equation} \label{eq:boundfirstquantity}
		\Ecal\Big( \mathbbm{1}_{\{x \leftrightarrow y\}} \,  \frac{1}{\tilde{n}_x \tilde{n}_y} \Big) \leq 2N \, \langle S_x^1 S_y^1 S_x^2 S_y^2 \rangle.
	\end{equation}
Besides, by \cite[Theorem~3.14]{QuitmannTaggi}, and Lemma~\ref{lemma:local-time-beta}, we have that, for every~$x\in\Lambda_L$,
\begin{equation}
\label{eq:boundmomentloctime}
\Ecal\Big((\tilde{n}_x)^{2^{m}}\Big)
= \E\Big((\tilde{n}_x)^{2^{m}}\Big)
\leq
b\,,
\end{equation}
where~$b=b(m,\,2d,\,\beta,\,N,\,U_N^s)<\infty$ is given by Lemma~\ref{lemma:local-time-beta}, and
with slight abuse of notation we also set $\tilde{n}_o(w):=n_o(w)+\frac{N}{2}$ for any $w \in \widetilde{\Wcal}$.
	Rearranging the terms in  \eqref{eq:upperbound} and using  \eqref{eq:boundfirstquantity} and~\eqref{eq:boundmomentloctime} we deduce the lower bound of \eqref{eq:mainequation}, with
$$
c_1
=
\frac{1}{2N}
b^{-\frac{1}{2^{m-1}}}
\,.
$$
	This ends the proof of the proposition.
\end{proof}

\begin{proof}[Proof of Theorem \ref{theo:nomacroscopicloops}]
	The {result} follows from   Proposition \ref{prop:correlationdecay} and 
	\ref{prop:fromspinstoloops}.
\end{proof}

\section*{Acknowledgements} The authors thank the German Research Foundation (project number 444084038, priority program SPP2265) for financial support. AQ additionally thanks the German Research Foundation through IRTG 2544 for financial support.

\appendix
\section*{Notation}
\begin{center}
	\begin{tabular}{ l l }
		
		$\N$,  $\N_0$ &  strictly positive and non-negative integers respectively \\
		
		$\R^+$, $ \R^+_0$  & strictly positive  and non-negative real numbers  respectively \\

		$[n]$ & set of integers $\{1,2,\dots,n\}$ \\
		
	%	$G = ( V, E)$ 
	%	&  an undirected, simple, finite  graph (sometimes assumed to be bipartite) \\
		
	%	$U:\mathbb{N}_0\to\mathbb{R}_0^+$ 
	%	&  {weight function} \\
		
		$\mathcal{L}=\cup_{x\in V}\mathcal{L}(x,\,x)$
		& set of rooted oriented loops \\
		
		$\Omega=\cup_{n\ge 0}\mathcal{L}^n$
		& configuration space of the RWLS \\
		
		$\mathcal{P}_{G, U,  N, \beta}$ 
		& probability distribution of the RWLS on~$\Omega$, defined in~\eqref{eq:measure} \\

		$\mathcal{Z}_{G, U,  N, \beta}$ 
		& partition function related to the distribution~$\mathcal{P}_{G, U,  N, \beta}$, given by~\eqref{eq:partition} \\
		
		$\mathcal{E}_{G, U,  N, \beta}$ 
		& expectation relative to the distribution~$\mathcal{P}_{G, U,  N, \beta}$ \\

			   $ \boldsymbol{e}_i$  & cartesian vector, with $i \in \{1, \ldots d\}$ or   $i \in \{1, \ldots d+1\}$ \\ 
		
		$e \in \mathcal{E}$ or $\{x,y\} \in \mathcal{E}$ & undirected edges \\
		
		$(x,y) \in \mathcal{E}$ &  edge directed from $x$ to $y$ \\

		 $N \in \mathbb{N}_{>0}, \, \, \lambda \in \R^+$& respectively number of colours, and  edge-parameter \\

		$m = (m_e)_{e \in E}$ 
		& link configuration, with  $m_e\in\mathbb{N}_0$ corresponding  number of links\\
		& on the edge $e$  \\
		
		$\mathcal{M}=(\mathbb{N}_0)^E$ 
		& set of link configurations on the graph~$G$ \\

		$\pi = (\pi_x)_{x \in V}$ & pairings, with $\pi_x$ pairing the links touching $x$ or leaving them unpaired \\
		
		%$g=(g_x)_{x \in \Vcal}$ & configuration of ghost pairings \\

		%$\mathcal{C}_{\mathcal{G}}(m)$ & the set of colourings for  $m\in\mathcal{M}_{\mathcal{G}}$ 
		
		$\mathcal{P}_{x}(m)$ 
		& set of pairing configurations for $m\in\mathcal{M}$ at one vertex~$x$ \\
		
		\smash{$\mathcal{P}(m)=\prod_{x\in V}\mathcal{P}_{x}(m)$ }
		& set of pairing configurations for $m\in\mathcal{M}$ \\

		$\mathcal{W}$ 
		&  set of configurations $w= (m,\pi)$ of the RPM on $G$\\ % \in \mathcal{W}_G$ consisting of~$m\in\mathcal{M}_G$ and~$\pi\in\mathcal{P}_{G}(m)$ \\
		
		\smash{$\widetilde{\Wcal}$ }
		& set of configurations of the RPM with no unpaired links \\
		
		$\alpha(w)$
		& number of paths in a RPM configuration~$w \in \mathcal{W}$ \\
		
		$\P_{G,U,N,\beta}$
		& probability distribution of the RPM on~\smash{$\widetilde{\Wcal}$}, defined by~\eqref{eq:RPMprobabilitymeasure} \\
		
	$\Z_{G,U,N,\beta}$
		& partition function related to the distribution~$\P_{G,U,N,\beta}$ \\
		
		$\E_{G,U,N,\beta}$
		& expectation under the RPM measure~$\P_{G,U,N,\beta}$ \\

%		$\gamma(\ell)$
%		& equivalence class of a rooted oriented loop~$\ell\in\mathcal{L}$ \\
		
%		$\Sigma(\mathcal{L})$
	%	& set of equivalence classes of rooted oriented loops \\
		
	%	$k_{\gamma}(\omega)$
	%	& multiplicity number of the equivalence class~$\gamma\in\Sigma(\mathcal{L})$ %in~$\omega\in\Omega$ \\

	%	$k_e(\omega)$
	%	& multiplicity number of the equivalence class of loops of length two\\
	%	& on the edge~$e$ in~$\omega\in\Omega$ \\
		
	%	$\tilde{k}_\gamma(w)$
	%	& multiplicity number of the equivalence class~$\gamma\in\Sigma(\mathcal{L})$ in a \\
	%	&RPM configuration~\smash{$w\in\widetilde{\Wcal}$} \\
		
	%	$\tilde{k}_e(w)$
	%	& multiplicity number of the equivalence class of loops of length two\\
	%	& on the edge~$e$ in~\smash{$w\in\widetilde{\Wcal}$} \\
		
	%	$\mathcal{R}_m$
	%	& set of rooted oriented linked loops relative to a link configuration~$m\in\Mcal$ \\
		
	%	$\chi(L)$
	%	& equivalence class of a rooted oriented linked loop~$L\in\Rcal_m$ \\
		
	%	$\Sigma(\Rcal_m)$
	%	& set of equivalence classes of rooted oriented linked loops \\
		
		%$n_x$ & number of pairings at $x$ \\
		
		%$u_x$ & number of unpaired endpoints of links at $x$ \\
		
		%$\Omega_\Gcal$ & set of configurations of the RWLS in $\mathcal{G}$ \\
		
		%$\Lcal_\Gcal$ & set of rooted oriented loops in $\Gcal$ \\
		
		%$\Sigma(\Wcal_\Gcal)$,  $\Sigma(\Omega_\Gcal)$, $ \Sigma(\Lcal_\Gcal)$ & set of equivalence classes  \\

		%$\delta(\ell)$ & stretch-factor of $\ell \in \Lcal_\Gcal$ \\
		
		$J(\ell)$ 
		& multiplicity of a rooted oriented loop $\ell \in \Lcal$ \\
		
		$\delta(\ell)$
		& stretch factor of a rooted oriented loop $\ell \in \Lcal$ \\

	\end{tabular}
	
\end{center}
\end{document}